\documentclass{article}
\usepackage[english]{babel}
\usepackage{amsmath,amsthm,amssymb,amsfonts}
\usepackage[a4paper,margin=2cm]{geometry}
\newcommand{\din}{\delta_{in}}
\newcommand{\dout}{\delta_{out}}
\newcommand{\cin}{c_{in}}
\newcommand{\cout}{c_{out}}
\newcommand{\ocin}{{\overline{c}_{in}}}
\newcommand{\ocout}{{\overline{c}_{out}}}
\newcommand{\of}[1]{{\overline{f_{#1}}}}
\newcommand{\E}{{\mathbf E}}
\newtheorem{theorem}{Theorem}
\newtheorem{lemma}{Lemma}
\newcommand{\Beta}{\mathrm{B}}
\title{The degree distribution and the number of edges between nodes of given degrees in directed scale-free graphs}

\author{Evgeniy A. Grechnikov\footnote{Research division in Yandex, Moscow.}}

\date{}
\begin{document}
\maketitle
\begin{abstract}

In this paper, we study some important statistics of the random graph in the
directed preferential attachment model
introduced by B. Bollob\'as, C. Borgs, J. Chayes and O. Riordan.
First, we find a new asymptotic formula for the expectation of the number $ n_{in}(d,t) $ 
of nodes of a given in-degree $ d $ in a graph in this model with $t$ edges, which covers all possible degrees.
The out-degree distribution in the model is symmetrical to the in-degree distribution.
Then we prove tight concentration for $n_{in}(d,t)$ while $d$ grows up to the moment when $n_{in}(d,t)$ decreases to $\ln^2 t$;
if $d$ grows even faster, $n_{in}(d,t)$ is zero \textbf{whp}.
Furthermore, we study a 
more complicated statistic of the graph: $ X(d_1,d_2,t) $ is the total number of edges from a vertex of out-degree
$d_1$ to a vertex of in-degree $d_2$. We also find an asymptotic formula for the expectation of $ X(d_1,d_2,t) $ and prove a tight 
concentration result.

\end{abstract}

\section{Introduction}
The real world has many interesting structures which can be thought of as graphs.
A typical example is the World Wide Web: one can consider web pages to be vertices
of a graph and hyperlinks to be edges. One of productive methods for
studying these graphs involves investigation of a suitable random graph model.

First models of random graphs were constructed and investigated
long ago. Classical models and results are systematized, for example,
in \cite{bollobasbook} and \cite{Luc}. However, they are not suitable
for approximation of dynamically changing and non-uniform networks.
In particular, the degree sequences of the graphs in these models are very far 
from those observed in reality. 

Recently other models of random graphs were constructed to more closely
match the growth of real networks. One of the first descriptions 
of such a model belongs
to the article \cite{barabasi} by Barab\'asi and Albert. The graph in this model
grows sequentially in discrete steps. Each step adds one vertex and several edges
that connect a new vertex with existing ones according to the ``preferential attachment''
rule: the probability of an existing vertex to receive a new edge
depends on the current degree of this node, so
more ``popular'' nodes are more attractive for new edges. In \cite{barabasi},
the probability is proportional to the current degree; \cite{barabasi} gives
some heuristic arguments suggesting that the number of vertices with degree $d$
decreases proportional to $d^{-3}$. The same quantity in real networks decreases
proportional to $d^{-\gamma}$ with different $\gamma$
for different networks, following the so called ``power law''.
Later, Bollob\'as, Riordan et al. proposed an explicit model
in \cite{bollobas1} based on the preferential attachment rule.
The model of \cite{bollobas1} resolves some ambiguities of \cite{barabasi};
also, \cite{bollobas1} rigorously proves a theorem about degree sequence.

In \cite{dorogovtsev} and \cite{drinea} two groups of researchers independently
proposed to add to the model one more parameter --- an ``initial attractiveness''
of a node which is a positive constant not depending on the degree.
Equivalently, the probability in the proposed model is a linear function
in the degree. Articles \cite{dorogovtsev} and \cite{drinea} use some heuristic
arguments to show that the model allows to obtain the power law for degree sequence
with any exponent less than $-2$. Buckley and Osthus in \cite{buckley} formalized
the model and rigorously proved the power law when all the parameters are natural numbers
and the degree grows slowly compared to the number of vertices.
We analyzed this model in \cite{grechnik}, removing the restriction on degree and
considering also the quantity similar to $X(t,d_1,d_2)$ in the current work.

Cooper and Frieze analyzed a quite general model in \cite{frieze}.
There are two different procedures for updating the graph in this model: adding a new vertex with
several edges and adding several edges to an already existing vertex. At every step, one of
those procedures is selected at random independently of other steps.

All models mentioned above yield essentially an undirected graph;
although edges have a natural direction, the out-degree sequence is unnatural.
In particular, in models \cite{bollobas1} and \cite{buckley} all vertices have the same
out-degree. Two models of directed graphs were suggested in \cite{chayes} and \cite{cooper}.
Both models have three different procedures
for updating the graph and select one of them at random for every step independently.
Both models mix preferential attachment with uniform random selection of source and target vertices,
but the details of mixing are different: \cite{chayes} follows \cite{buckley} and assigns probabilities that are proportional
to a linear shift of vertex degree, while \cite{cooper} follows \cite{frieze} and uses weighted sum
of probability of preferential attachment (that is proportional to vertex degree) and probability of uniform selection.
Also, model of \cite{chayes} always creates a vertex with one edge (incoming or outgoing) in the name of simplicity,
while model of \cite{cooper} allows an arbitrary finite distribution for number of edges created with a new vertex
in the name of generality. These models tend to give similar results, though.

\section{The model and formulation of results}
We analyze the model $\mathcal G(t)$ of a random graph introduced in \cite{chayes}.
Denote the in-degree of a vertex $v$ as $d_{in}(v)$ and the out-degree of a vertex $v$ as $d_{out}(v)$.
\begin{itemize}
\item There are 6 parameters $\alpha\in[0,1]$, $\beta\in[0,1]$, $\gamma\in[0,1]$,
$\alpha+\beta+\gamma=1$, $\din\ge0$, $\dout\ge0$, a graph $G_0$.
\item The graph $G_0$ should contain at least one vertex. If $\din=0$ or $\dout=0$,
it should also contain at least one edge.
\item There is also time $t\in\mathbb{Z}$, $t\ge t_0$, where $t_0$ is the number of
edges in $G_0$.
\item The probability space at time $t_0$ contains only one graph $G_0$.
\item Given a random graph $G$ at time $t$, a random graph at time $t+1$
is constructed from $G$ by one of the following processes.
\begin{enumerate}
\item[$\dagger$] With probability $\alpha$, add a new vertex $v$ and an edge from $v$
to one of vertices in $G$. The target vertex is selected randomly, a vertex
$w\in G$ is selected with probability $\frac{d_{in}(w)+\din}{t+\din n}$.
\item[$\dagger\dagger$] With probability $\beta$, add an edge from a random vertex $v\in G$ to
a random vertex $w\in G$. A vertex $v$ is selected with
probability $\frac{d_{out}(v)+\dout}{t+\dout n}$,
a vertex $w$ is selected independently of $v$ with probability
$\frac{d_{in}(w)+\din}{t+\din n}$.
\item[$\ddagger$] With probability $\gamma$, add a new vertex $w$ and an edge from one
of vertices in $G$ to $w$. The source vertex is selected randomly, a vertex
$v\in G$ is selected with probability $\frac{d_{out}(v)+\dout}{t+\dout n}$.
\end{enumerate}
\item It is easy to see that a random graph at time $t$ has exactly $t$ edges
and a random number of vertices concentrated around $(\alpha+\gamma)t$.
\end{itemize}
We assume that $\alpha+\gamma>0$; otherwise number of vertices would not change over time.

Let $\ocin=\frac{1-\gamma}{1+\din(\alpha+\gamma)}$ and
$\ocout=\frac{1-\alpha}{1+\dout(\alpha+\gamma)}$.

Our first topic is the in-degree sequence in this model. Obviously, the out-degree sequence
has the same structure with exchanging $\alpha\leftrightarrow\gamma$ and $\din\leftrightarrow\dout$.
Let $n_{in}(t,d)$ be a number of vertices with in-degree $d$ in a random graph $G\in\mathcal G(t)$.
There are two special cases for the in-degree sequence that are not interesting.
If $\alpha+\beta=0$, then $\ocin=0$ and every vertex not in $G_0$ has in-degree 1.
If $\gamma+\din=0$, then $\ocin=1$ and every vertex not in $G_0$ has in-degree 0.
Otherwise, $0<\ocin<1$.

Define
\begin{eqnarray*}
\of0&=&\frac\alpha{1+\din\ocin},\\
\of d&=&C_{in}\frac{\Gamma(d+\din)}{\Gamma\left(d+\din+1+\frac1\ocin\right)}\mbox{ for }d\ge1,\\
C_{in}&=&\frac{\Gamma\left(\din+\frac1\ocin\right)}{\Gamma(1+\din)}\frac{1-\ocin}{\ocin^2}.
\end{eqnarray*}
Then $\of d\sim C_{in}d^{-1-\frac1\ocin}$ as $d$ grows due to a standard result about the gamma-function
(e.g. \cite[6.1.47]{specfunc}).

The article \cite{chayes} proves that $n_{in}(t,d)=\of dt(1+o_d(1))$ (with probability tending to 1 as $t\to\infty$),
i.e. the power law when $d$ is fixed.
The article \cite{cooper} proves that $n_{in}(t,d)=\of dt\left(1+O\left(\frac1{\sqrt{\log t}}\right)\right)$
(again, with probability tending to 1 as $t\to\infty$) for $d\le\min\left\{t^{\ocin/3},\frac{t^{1/6}}{\log^2t}\right\}$
(in the model of \cite{cooper} after an appropriate mapping of parameters); for the maximal value of $d$ covered
by \cite{cooper} we have $n_{in}(t,d)\gg t^{\max\left\{\frac{2-\ocin}3,\frac{5-1/\ocin}6\right\}}\ge\sqrt t$.
Our results are valid for all possible degrees,
give the concentration up to the moment where $n_{in}(t,d)$ becomes $O(\ln^2 t)$
and have much better remainder term for values of $d$ covered by \cite{chayes} and \cite{cooper}.

\begin{theorem} Let $\alpha+\beta>0$ and $\gamma+\din>0$. Let $\varepsilon>0$ be arbitrarily small. Then
$$
\E n_{in}(t,d)=\of dt+O\left((d+1)^{-1+\varepsilon}\right).
$$
\end{theorem}
\begin{theorem} Let $\alpha+\beta>0$ and $\gamma+\din>0$. Let $\varepsilon>0$ be arbitrarily small. Let $d=d(t)$. Then
$$
|n_{in}(t,d)-\of d t|\le\left(\sqrt{\of d t}+(d+1)^{-\frac12+\varepsilon}\right)\ln t
$$
with probability tending to 1 as $t\to\infty$.
\end{theorem}

When $d=o\left(\left(\frac{t}{\ln^2 t}\right)^{\frac\ocin{1+\ocin}}\right)$, that is equivalent
to $\ln^2t=o\left(\of d t\right)$,
Theorem 2 implies the equivalence (with probability tending to 1 as $d,t\to\infty$)
$$
n_{in}(t,d)\sim C_{in}d^{-1-\frac1\ocin}t.
$$
When $t^{\frac\ocin{1+\ocin}}=o(d)$, Theorem 1 implies $\E n_{in}(t,d)=o(1)$; since $n_{in}(t,d)$
is an integer number by definition, $n_{in}(t,d)=0$ (again, with probability tending to 1 as $d,t\to\infty$).
Thus, we have an almost entire picture of what happens to $ n_{in}(t,d) $. 

Our second topic concerns the expected number of edges (or probability of an edge) between two given vertices.
Suppose that we know their in- and out-degrees, but nothing else. For a fixed vertex, in-degree and out-degree
are essentially independent due to the construction (although they both tend to grow with age). Thus,
we use only out-degree $d_1$ of the potential source and in-degree $d_2$ of the potential target.
Theorems 1 and 2 give the total number of vertices with in-degree $d_2$ and the total number of vertices
with out-degree $d_1$. Thus, we estimate the total number of edges between vertices of given degrees.
For this topic, we consider $d_1$ and $d_2$ to be fixed, not growing with $t$, to simplify
calculations somewhat.

More precisely, we define the random variable $X(t,d_1,d_2)$ as the total number of edges
in a random graph $G\in\mathcal G(t)$
with the following property: the out-degree of the source vertex is $d_1\ge1$,
the in-degree of the target vertex is $d_2\ge1$, source and target vertices are different
(i.e. for $d_1=d_2$ we do not count loops).

Define $$\kappa(c_1,c_2,r,x)=\int_0^x z^{c_1-1}dz\int_0^\infty\tau^{c_1r+c_2-1}e^{-\tau-z\tau^r}d\tau=
x^{c_1}\int_0^1 z^{c_1-1}dz\int_0^\infty\tau^{c_1r+c_2-1}e^{-\tau-xz\tau^r}d\tau$$ for $c_1>0$, $c_2>0$, $r>0$, $x\ge0$.
As a function of $x$, it monotonically increases from $0$ to $\Gamma(c_1)\Gamma(c_2)$ as $x$ grows from $0$ to $\infty$.
The asymptotic behaviour when $x\to0$ is given by $\kappa(c_1,c_2,r,x)=\left(\frac{\Gamma(c_1r+c_2)}{c_1}+O(x)\right)x^{c_1}$.

Note: in the special case $r=1$ the inner integral can be calculated and $\kappa$ becomes the incomplete beta-function:
$\kappa(c_1,c_2,1,x)=\Gamma(c_1+c_2)\Beta\left(\frac x{1+x};c_1,c_2\right)$. We won't use this fact, but it can be useful
to conceive the function $\kappa$.

\begin{theorem} Let $\alpha>0$, $\beta>0$, $\gamma>0$, $\din>0$, $\dout>0$, $d_1\ge2$, $d_2\ge2$.
Then $\E X(t,d_1,d_2)=c_X(d_1,d_2)t+O_{d_1,d_2}(1)$, where
$$c_X(d_1,d_2)=c_{X1}(d_1,d_2)+c_{X2}(d_1,d_2)+c_{X3}(d_1,d_2)$$ and
\begin{equation*}
c_{X1}(d_1,d_2)=d_1^{-\frac1\ocout}d_2^{-\frac\ocout\ocin-1}
\sum_{i=0}^1 A_i
\kappa\left(\dout+\frac1\ocout+i,\din+\frac\ocout\ocin+1,\frac\ocout\ocin,\frac{d_1}{d_2^{\ocout/\ocin}}\right)
\left(1+O\left(\frac1{d_1}+\frac1{d_2}\right)\right),
\end{equation*}
$A_0=\frac{\gamma^2}{(1+\dout(\alpha+\gamma))\ocin\ocout\Gamma(\dout)\Gamma(1+\din)}$,
$A_1=\frac{\alpha\gamma}{(1+\dout(\alpha+\gamma))\ocin\ocout\Gamma(1+\dout)\Gamma(1+\din)}$,
\begin{multline*}
c_{X2}(d_1,d_2)=
d_1^{-\frac\ocin\ocout-1}d_2^{-\frac1\ocin}\sum_{i=0}^1 B_i
\kappa\left(\din+\frac1\ocin+i,\dout+\frac\ocin\ocout+1,\frac\ocin\ocout,\frac{d_2}{d_1^{\ocin/\ocout}}\right)
\left(1+O\left(\frac1{d_1}+\frac1{d_2}\right)\right),
\end{multline*}
$B_0=\frac{\alpha^2}{(1+\din(\alpha+\gamma))\ocin\ocout\Gamma(1+\dout)\Gamma(\din)}$,
$B_1=\frac{\alpha\gamma}{(1+\din(\alpha+\gamma))\ocin\ocout\Gamma(1+\dout)\Gamma(1+\din)}$.

If $\ocin+\ocout\ne1$, then
\begin{multline*}
c_{X3}(d_1,d_2)=d_1^{-\frac1\ocout}d_2^{-\frac1\ocin}\sum_{i=0}^1\sum_{j=0}^1\frac{C_{ij}}{1-\ocin-\ocout}
\\\times
\Bigg(d_2^{\frac{1-\ocin-\ocout}\ocin}\kappa\left(\dout+\frac1\ocout+i,\din+\frac\ocout\ocin+1+j,\frac\ocout\ocin,\frac{d_1}{d_2^{\ocout/\ocin}}\right)\left(1+O\left(\frac1{d_1}+\frac1{d_2}\right)\right)
\\+d_1^{\frac{1-\ocin-\ocout}\ocout}\kappa\left(\din+\frac1\ocin+j,\dout+\frac\ocin\ocout+1+i,\frac\ocin\ocout,\frac{d_2}{d_1^{\ocin/\ocout}}\right)\left(1+O\left(\frac1{d_1}+\frac1{d_2}\right)\right)
\\-\Gamma\left(\dout+\frac1\ocout+i\right)\Gamma\left(\din+\frac1\ocin+j\right)
+O\left(\frac1{d_1}+\frac1{d_2}\right)\Bigg),
\end{multline*}
$C_{00}=\frac{\beta\alpha\gamma}{(1+\din(\alpha+\gamma))(1+\dout(\alpha+\gamma))\ocin\ocout\Gamma(\dout)\Gamma(\din)}$,\\
$C_{01}=\frac{\beta\gamma^2}{(1+\din(\alpha+\gamma))(1+\dout(\alpha+\gamma))\ocin\ocout\Gamma(\dout)\Gamma(\din+1)}$,\\
$C_{10}=\frac{\beta\alpha^2}{(1+\din(\alpha+\gamma))(1+\dout(\alpha+\gamma))\ocin\ocout\Gamma(\dout+1)\Gamma(\din)}$,\\
$C_{11}=\frac{\beta\alpha\gamma}{(1+\din(\alpha+\gamma))(1+\dout(\alpha+\gamma))\ocin\ocout\Gamma(\dout+1)\Gamma(\din+1)}$.

If $\ocin+\ocout=1$, then
\begin{multline*}
c_{X3}(d_1,d_2)=d_1^{-\frac1\ocout}d_2^{-\frac1\ocin}
\sum_{i=0}^1\sum_{j=0}^1 C_{ij}\Bigg(
\frac1\ocin\kappa\left(\dout+\frac1\ocout+i,\din+\frac1\ocin+j,\frac\ocout\ocin,\frac{d_1}{d_2^{\ocout/\ocin}}\right)\ln d_2
\\-\frac1\ocin\frac{\partial\kappa}{\partial c_2}\left(\dout+\frac1\ocout+i,\din+\frac1\ocin+j,\frac\ocout\ocin,\frac{d_1}{d_2^{\ocout/\ocin}}\right)\\
+\frac1\ocout\kappa\left(\din+\frac1\ocin+j,\dout+\frac1\ocout+i,\frac\ocin\ocout,\frac{d_2}{d_1^{\ocin/\ocout}}\right)\ln d_1
\\-\frac1\ocout\frac{\partial\kappa}{\partial c_2}\left(\din+\frac1\ocin+j,\dout+\frac1\ocout+i,\frac\ocin\ocout,\frac{d_2}{d_1^{\ocin/\ocout}}\right)
+O\left(\frac{\ln d_2}{d_1}+\frac{\ln d_1}{d_2}\right)
\Bigg)
\end{multline*}
(where $C_{ij}$ are same as in the previous case).

If $d_1$ and $d_2$ both grow to infinity such that $d_1^\ocin/d_2^\ocout\to0$, then
$$
c_X(d_1,d_2)\sim\begin{cases}
D_{-}d_1^{-\frac{\ocin+\ocout}\ocout}d_2^{-\frac1\ocin},&\ocin+\ocout<1\\
D_{0}d_1^{-\frac1\ocout}d_2^{-\frac1\ocin}\ln d_1,&\ocin+\ocout=1\\
D_{+}d_1^{-\frac1\ocout}d_2^{-\frac1\ocin},&\ocin+\ocout>1,
\end{cases}
$$
\begin{multline*}
D_{-}=\left(B_0\Gamma\left(\din+\frac1\ocin\right)+B_1\Gamma\left(\din+\frac1\ocin+1\right)\right)\Gamma\left(\dout+\frac\ocin\ocout+1\right)
\\+\sum_{i=0}^1\sum_{j=0}^1\frac{C_{ij}}{1-\ocin-\ocout}\Gamma\left(\din+\frac1\ocin+j\right)\Gamma\left(\dout+\frac\ocin\ocout+1+i\right);
\end{multline*}
$$
D_{0}=\sum_{i=0}^1\sum_{j=0}^1\frac{C_{ij}}{\ocout}\Gamma\left(\dout+\frac1\ocout+i\right)\Gamma\left(\din+\frac1\ocin+j\right);
$$
$$
D_{+}=\sum_{i=0}^1\sum_{j=0}^1\frac{C_{ij}}{\ocin+\ocout-1}\Gamma\left(\dout+\frac1\ocout+i\right)\Gamma\left(\din+\frac1\ocin+j\right).
$$
If $d_1$ and $d_2$ both grow to infinity such that $d_1^\ocin/d_2^\ocout\to\infty$, then
$$
c_X(d_1,d_2)\sim\begin{cases}
D'_{-}d_1^{-\frac1\ocout}d_2^{-\frac{\ocin+\ocout}\ocin},&\ocin+\ocout<1\\
D'_{0}d_1^{-\frac1\ocout}d_2^{-\frac1\ocin}\ln d_2,&\ocin+\ocout=1\\
D_{+}d_1^{-\frac1\ocout}d_2^{-\frac1\ocin},&\ocin+\ocout>1,
\end{cases}
$$
\begin{multline*}
D'_{-}=\left(A_0\Gamma\left(\dout+\frac1\ocout\right)+A_1\Gamma\left(\dout+\frac1\ocout+1\right)\right)\Gamma\left(\din+\frac\ocout\ocin+1\right)
\\+\sum_{i=0}^1\sum_{j=0}^1\frac{C_{ij}}{1-\ocin-\ocout}\Gamma\left(\din+\frac\ocout\ocin+1+j\right)\Gamma\left(\dout+\frac1\ocout+i\right);
\end{multline*}
$$
D'_{0}=\sum_{i=0}^1\sum_{j=0}^1\frac{C_{ij}}{\ocin}\Gamma\left(\dout+\frac1\ocout+i\right)\Gamma\left(\din+\frac1\ocin+j\right).
$$
\end{theorem}
Note: one can see from the proof that $c_{X1}$, $c_{X2}$, $c_{X3}$ have their own physical sense:
they give the fraction of edges produced by $(\ddagger)$, $(\dagger)$ and $(\dagger\dagger)$
respectively (relative to $t$, the total number of edges).

\begin{theorem} Let $d_1$ and $d_2$ be fixed. In conditions of the previous theorem
$$
|X(t,d_1,d_2)-\E X(t,d_1,d_2)|<\sqrt{t}\ln t
$$
with probability tending to 1 as $t\to\infty$.
\end{theorem}

The average value of the number of edges from a vertex with out-degree $d_1$
to a vertex with in-degree $d_2$ is $\frac{X(t,d_1,d_2)}{n_{out}(t,d_1)n_{in}(t,d_2)}$.
Since $n_{in}(t,d)$ and $X(t,d_1,d_2)$ are tightly concentrated around their expectations,
the main term of the ratio is given by Theorems 1 and 3. In particular,
if $\ocin+\ocout>1$, the average number of edges is proportional to $\frac{d_1d_2}{t}$,
and if $\ocin+\ocout<1$, it grows somehow between $\frac1t d_1^{\frac{1-\ocin}\cout}d_2$
and $\frac1t d_1d_2^{\frac{1-\ocout}\ocin}$
(tending to the first expression when $d_1$ is small compared to $d_2$
and to the second one when $d_1$ is large; Theorem 3 describes intermediate cases in detail).

\section{Expected number of vertices with the given degree}
\begin{proof}[Proof of Theorem 1]
Let $n_0=\#G_0$, $N=n-n_0$, $T=t-t_0$, $A_{in}=t_0+\din n_0$.
Let $E_d(T,N)=\E(n_{in}(T+t_0,d)|\#G=N+n_0)$.
Obviously,
$$
\E n_{in}(T+t_0,d)=\sum_{N=0}^{T}E_d(T,N)\Pr(\#G=N+n_0)=
\sum_{N=0}^{T}E_d(T,N)(\alpha+\gamma)^N\beta^{T-N}\binom{T}{N}.
$$
If $\beta=0$, the conditional expectation is defined only when $N=T$;
in this case, we will define $E_d(T,N)$ for $N\ne T$ later, the formula
holds for any definition.

For $x\in[0,1]$, let $\cin=\cin(x)=\frac{1-\frac\gamma{\alpha+\gamma}x}{1+\din x}\in[0,1]$ and
$\cout=\cout(x)=\frac{1-\frac\alpha{\alpha+\gamma}x}{1+\dout x}\in[0,1]$. Note that
$\ocin=\cin(\alpha+\gamma)$, $\ocout=\cout(\alpha+\gamma)$.
Let $p_{out,0}=p_{in,1}=\frac\gamma{\alpha+\gamma}$, $p_{out,1}=p_{in,0}=\frac\alpha{\alpha+\gamma}$.

\begin{lemma}\label{expect1} Let $\alpha+\beta>0$, $\gamma+\din>0$, $\varepsilon>0$. For $T\ge1$, $d\ge0$, $0\le N\le T$,
$$
E_d(T,N)=Tf_d\left(\frac NT\right)+O\left(\frac1{1+(d+\din+1)^{1-\varepsilon}N^2/T^2}\right),
$$
where
$$
f_0(x)=\frac{ax}{1+\cin\din},\qquad a=\frac{\alpha}{\alpha+\gamma},
$$
$$
f_d(x)=\frac{1-\cin}{(1+\cin\din)(1+\cin(\din+1))}\prod_{i=2}^d\frac{\cin(\din+i-1)}{1+\cin(\din+i)}\mbox{ for }d\ge1.
$$
If $\din>0$ and either $\alpha>0$ or $x<1$, then 
$$
f_d(x)=x\sum_{i=0}^1 [d\ge i]p_{in,i}\frac{\Gamma\left(i+\din+\frac1\cin\right)}{\cin\Gamma(i+\din)}
\frac{\Gamma(d+\din)}{\Gamma\left(d+1+\din+\frac1\cin\right)}.
$$
\end{lemma}
\begin{proof} Without loss of generality assume $\varepsilon<1$.

First, derive a recurrent equation for $E_d(T,N)$. Let $\deg_{in,T}(i)$ denote
in-degree of a vertex $i$ in a random graph $G_T\in\mathcal G(T+t_0)$.
$$
E_d(T+1,N)=\sum_{i=1}^{N+n_0}\Pr(\deg_{in,T+1}(i)=d|\#G_{T+1}=N+n_0)=
\sum_{i=1}^{N+n_0}\frac{\Pr(\deg_{in,T+1}(i)=d,\#G_{T+1}=N+n_0)}{\Pr(\#G_{T+1}=N+n_0)}.
$$
The denominator is easy to calculate: $\Pr(\#G_{T+1}=N+n_0)=(\alpha+\gamma)^N\beta^{T+1-N}\binom{T+1}{N}$.

Let $G_{T+1}$ be a random graph from $\mathcal G(T+t_0+1)$ with $N+n_0$ vertices.
The numerator can be expressed as a sum $\Pr(...,\dagger)+\Pr(...,\ddagger)+\Pr(...,\dagger\dagger)$
of probabilities with additional condition that $G_{T+1}$ is constructed from $G_T\in\mathcal G(T+t_0)$
using ($\dagger$), ($\ddagger$), ($\dagger\dagger$) respectively.

\begin{itemize}
\item If $G_{T+1}$ is constructed from $G_T\in\mathcal G(T+t_0)$ using ($\dagger$), then
$N>0$, $G_T$ has $N+n_0-1$ vertices. The last vertex in $G_{T+1}$ has in-degree 0.
For $i\in G_T$,
\begin{multline*}
\Pr(\deg_{in,T+1}(i)=d,\#G_{T+1}=N+n_0,\dagger)\\=\alpha
\Pr(\deg_{in,T}(i)=d,\#G_T=N+n_0-1)\left(1-\frac{d+\din}{T+\din N+A_{in}-\din}\right)\\
+\alpha \Pr(\deg_{in,T}(i)=d-1,\#G_T=N+n_0-1)\frac{d-1+\din}{T+\din N+A_{in}-\din}.
\end{multline*}
Note that
$$
\frac{\Pr(\#G_T=N+n_0-1)}{\Pr(\#G_{T+1}=N+n_0)}=
\frac{(\alpha+\gamma)^{N-1}\beta^{T+1-N}\binom{T}{N-1}}{(\alpha+\gamma)^N\beta^{T+1-N}\binom{T+1}{N}}=
\frac{N}{(\alpha+\gamma)(T+1)}.
$$
\item If $G_{T+1}$ is constructed from $G_T\in\mathcal G(T+t_0)$ using ($\dagger\dagger$),
then $N<T+1$, $G_T$ has $N+n_0$ vertices. For each of them
\begin{multline*}
\Pr(\deg_{in,T+1}(i)=d,\#G_{T+1}=N+n_0,\dagger\dagger)\\=\beta
\Pr(\deg_{in,T}(i)=d,\#G_T=N+n_0)\left(1-\frac{d+\din}{T+\din N+A_{in}}\right)\\
+\beta \Pr(\deg_{in,T}(i)=d-1,\#G_T=N+n_0)\frac{d-1+\din}{T+\din N+A_{in}}.
\end{multline*}
Note that
$$
\frac{\Pr(\#G_T=N+n_0)}{\Pr(\#G_{T+1}=N+n_0)}=
\frac{(\alpha+\gamma)^{N}\beta^{T-N}\binom{T}{N}}{(\alpha+\gamma)^N\beta^{T+1-N}\binom{T+1}{N}}=
\frac{T+1-N}{\beta(T+1)}.
$$
\item If $G_{T+1}$ is constructed from $G_T\in\mathcal G(T+t_0)$ using ($\ddagger$),
then $N>0$, $G_T$ has $N+n_0-1$ vertices. The last vertex in $G_{T+1}$ has in-degree 1.
For $i\in G_T$,
$$\Pr(\deg_{in,T+1}(i)=d,\#G_{T+1}=N+n_0,\ddagger)=\gamma \Pr(\deg_{in,T}(i)=d,\#G_T=N+n_0-1).$$
\end{itemize}
Thus,
\begin{multline}\label{recur1}
E_d(T+1,N)=[N>0,d=0]\alpha\frac{N}{(\alpha+\gamma)(T+1)}\\
+[N>0]\alpha E_d(T,N-1)\left(1-\frac{d+\din}{T+\din N+A_{in}-\din}\right)\frac{N}{(\alpha+\gamma)(T+1)}\\
+[N>0]\alpha E_{d-1}(T,N-1)\frac{d-1+\din}{T+\din N+A_{in}-\din}\frac{N}{(\alpha+\gamma)(T+1)}\\
+[N<T+1]\beta E_d(T,N)\left(1-\frac{d+\din}{T+\din N+A_{in}}\right)\frac{T+1-N}{\beta(T+1)}\\
+[N<T+1]\beta E_{d-1}(T,N)\frac{d-1+\din}{T+\din N+A_{in}}\frac{T+1-N}{\beta(T+1)}\\
+[N>0,d=1]\gamma\frac{N}{(\alpha+\gamma)(T+1)}+[N>0]\gamma E_d(T,N-1)\frac{N}{(\alpha+\gamma)(T+1)}.
\end{multline}
Note that the right-hand side of \eqref{recur1} for $N=T+1$ depends only on $E_d(T,T)$
and not other values of $E_d(T,T)$. If $\beta=0$, define $E_d(T,N)$ for $N\ne T$ recurrently
so that \eqref{recur1} holds.

We need to establish some properties of $f_d(x)$.
A straightforward calculation shows that two definitions of $f_d(x)$ from the statement of the lemma
are indeed equivalent (when the second one is defined) and that $f_d(x)$ satisfy to the following recurrent equation:
\begin{equation}\label{fdrecur}
\left(\frac{1-x+ax}{1+\din x}(d+\din)+1\right)f_d(x)=\frac{1-x+ax}{1+\din x}(d+\din-1)f_{d-1}(x)+x\left([d=0]a+[d=1](1-a)\right),
\end{equation}
assuming $f_{-1}(x)=0$. It is obvious from the first definition of $f_d(x)$
that $f_d(x)$ is analytical for $x\in[0,1]$ and any fixed $d$.
Since
\begin{multline}\label{fdbound}
\prod_{i=2}^d\frac{\cin(\din+i-1)}{1+\cin(\din+i)}\le\prod_{i=2}^d\frac{\din+i-1}{\din+i+\frac1{\cin+\frac\varepsilon2}}
\le\prod_{i=2}^d\left(1+\frac1{\din+i+\frac1{\cin+\frac\varepsilon2}}\right)^{-1-\frac1{\cin+\frac\varepsilon2}}\\
=\left(\frac{d+\din+1+\frac1{\cin+\frac\varepsilon2}}{\din+2+\frac1{\cin+\frac\varepsilon2}}\right)^{-1-\frac1{\cin+\frac\varepsilon2}}
\le\left(\din+2+\frac2\varepsilon\right)^{1+\frac2\varepsilon}(d+\din+1)^{-1-\frac1{\cin+\frac\varepsilon2}}\\
=O\left((d+\din+1)^{-1-\frac1{\cin+\frac\varepsilon2}}\right),
\end{multline}
where the implied constant does not depend on $x$. Furthermore, for $d\ge3$ and $\cin\ne0$
we have\\ $\prod_{i=2}^d\frac{\cin(\din+i-1)}{1+\cin(\din+i)}=O\left(\cin^2(d+\din+1)^{-1-\frac1{\cin+\frac\varepsilon2}}\right)$, so
\begin{multline*}
\frac{d}{dx}\prod_{i=2}^d\frac{\cin(\din+i-1)}{1+\cin(\din+i)}=
\cin'\left(\prod_{i=2}^d\frac{\cin(\din+i-1)}{1+\cin(\din+i)}\right)\left(\sum_{i=2}^d\frac1{\cin(1+\cin(\din+i))}\right)\\
=O\left((d+\din+1)^{-1-\frac1{\cin+\frac\varepsilon2}}\sum_{i=2}^d\frac\cin{1+\cin(\din+i)}\right)
=O\left((d+\din+1)^{-1-\frac1{\cin+\frac\varepsilon2}}\sum_{i=2}^d\frac1{1+\din+i}\right)\\
=O\left((d+\din+1)^{-1-\frac1{\cin+\frac\varepsilon2}}\ln d\right)=O\left((d+\din+1)^{-1-\frac1{\cin+\varepsilon}}\right),
\end{multline*}
where the implied constant again does not depend on $x$; since $f_d'(x)$ is continuous,
we have $$f'_d(x)=O\left((d+\din+1)^{-1-\frac1{\cin+\varepsilon}}\right)$$
for any $x\in[0,1]$ (even in the point $x=1$, where
$\cin$ may be zero) and any $d$ (because obviously, $f'_0(x)=O(1)$, $f'_1(x)=O(1)$, $f'_2(x)=O(1)$).
Similarly, $$f''_d(x)=O\left((d+\din+1)^{-1-\frac1{\cin+\frac\varepsilon2}}\ln^2(d+2)\right)=O\left((d+\din+1)^{-1-\frac1{\cin+\varepsilon}}\right).$$
For this theorem, it is more convenient to use weaker bounds that do not depend on $x$:
$f_d(x)=O((d+1)^{-2+\varepsilon})$, $f'_d(x)=O((d+1)^{-2+\varepsilon})$, $f''_d(x)=O((d+1)^{-2+\varepsilon})$.
A stronger bound \eqref{fdbound} would not improve the final bound; we will need it for the next theorem, though.

The case $d=\din=0$ is special. In this case, \eqref{recur1} becomes
$$
E_0(T+1,N)=\frac{aN}{T+1}+[N>0]E_0(T,N-1)\frac{N}{T+1}+[N<T+1]E_0(T,N)\frac{T+1-N}{T+1}
$$
The solution is $E_0(T,N)=aN+E_0(0,0)$, it satisfies the condition of the lemma.

We use induction by $T$ and, for fixed $T$, by $d$
to prove the formula
\begin{equation}\label{lemm1main}
\left|E_d(T,N)-Tf_d\left(\frac NT\right)\right|\le C\left(d,\frac NT\right),\qquad 0\le N\le T,
\end{equation}
where $C(d,x)=\frac{C_0}{1+(d+\din+1)^{1-\varepsilon}x^2}$
and $C_0$ be some sufficiently large constant that will be determined later.

Select $T_0$ such that for any $T\ge T_0$ we have
$(T+t_0+\din+1)^{1-\varepsilon}\le T\frac{\varepsilon}{12(1+2\din+A_{in})}$
and $(T+t_0+\din+1)^{1-\varepsilon}\le T\frac{\din}{12(1+2\din+A_{in})}$ if $\din>0$.
There are only finitely many pairs $(T,N)$ with $T\le T_0$, so we can select $C_0$ such that
\eqref{lemm1main} holds for $T\le T_0$.

Before proceeding with induction, we need to establish \eqref{lemm1main} for $d\ge T+t_0$
and any $T\ge T_0$. In this case, $Tf_d(x)=O(Td^{-2+\varepsilon})=O(d^{-1+\varepsilon})$, so we need to prove that
$E_d(T,N)\le C\left(d,\frac NT\right)$ (and then increase $C_0$ to account for $Tf_d(x)$).
If $d>T+t_0$, it is trivial because $E_d(T,N)=0$.
There can be only one vertex with in-degree $d=T+t_0$, and it exists only
if $(\ddagger)$ has not been used after the first step and the same target vertex
was selected at every step; the probability of that given $\#G=N+n_0$ is not greater than
$\prod_{i=2}^{N}a\frac{T+t_0+\din}{T+t_0+\din i}=a^{N-1}\exp\sum_{i=2}^{N}\ln\left(1-\frac{\din(i-1)}{T+t_0+\din i}\right)
\le a^{N-1}\left(\exp\frac{\din N(N-1)}{2(T(1+\din)+t_0)}\right)^{-1}\le\frac{a^{N-1}}{1+\frac{\din N(N-1)}{2(T(1+\din)+t_0)}}$.
Since $\gamma+\din>0$, we have either $a<1$ or $\din>0$, in both cases $E_d(T,N)=O\left(\frac1{1+N^2/T}\right)$,
so \eqref{lemm1main} holds for $T\ge T_0$.

Now assume that \eqref{lemm1main} is proved for
some value of $T\ge T_0$ and consider the value $T+1$. Assume also that $d<T+1+t_0$.

Let $0\le N\le T+1$ and $x=\frac{N}{T+1}$. If $N\ne T+1$, the inductive hypothesis and the Taylor formula imply that
\begin{multline*}
E_d(T,N)=Tf_d\left(x+\frac{x}{T}\right)+\theta_1=Tf_d(x)+xf'_d(x)+\frac{x^2}{2T}f''_d(\xi)+\theta_1\\=
Tf_d(x)+xf'_d(x)+O\left(\frac{(d+1)^{-2+\varepsilon}}T\right)+\theta_1,
\end{multline*}
where $|\theta_1|\le C\left(d,x+\frac xT\right)$. Similarly
$$E_{d-1}(T,N)=Tf_{d-1}(x)+O\left((d+1)^{-2+\varepsilon}\right)+\theta_2$$
for $d\ge1$ with $|\theta_2|\le C\left(d-1,x+\frac xT\right)$. The same holds for $d=0$ with $\theta_2=0$.

If $N\ne0$, we have for the same reasons
\begin{multline*}
E_d(T,N-1)=Tf_d\left(x-\frac{1-x}T\right)+\theta_3=Tf_d(x)-(1-x)f'_d(x)+\frac{(1-x)^2}{2T}f''_d(\xi)+\theta_3\\=
Tf_d(x)-(1-x)f'_d(x)+O\left(\frac{(1-x)(d+1)^{-2+\varepsilon}}T\right)+\theta_3
\end{multline*}
and $E_{d-1}(T,N-1)=Tf_{d-1}(x)+O\left((1-x)(d+1)^{-2+\varepsilon}\right)+\theta_4$,
$|\theta_3|\le C\left(d,x-\frac{1-x}T\right)$, $\theta_4=0$ for $d=0$,
$|\theta_4|\le C\left(d-1,x-\frac{1-x}T\right)$ for $d\ge1$.

Now \eqref{recur1} becomes
\begin{multline}\label{recur2}
E_d(T+1,N)=[d=0]ax+[d=1](1-a)x\\+
x\left(Tf_d(x)-(1-x)f'_d(x)+O\left(\frac{(1-x)(d+1)^{-2+\varepsilon}}T\right)\right)\left(1-\frac{a(d+\din)}{T(1+\din x)}+O\left(\frac{a}{T}\right)\right)\\+
ax\left(Tf_{d-1}(x)+O\left((1-x)(d+1)^{-2+\varepsilon}\right)\right)\frac{d-1+\din}{T(1+\din x)}\left(1+O\left(\frac1T\right)\right)\\+
(1-x)\left(Tf_d(x)+xf'_d(x)+O\left(\frac{(d+1)^{-2+\varepsilon}}T\right)\right)\left(1-\frac{d+\din}{T(1+\din x)}+O\left(\frac{1}{T}\right)\right)\\+
(1-x)\left(Tf_{d-1}(x)+O\left((d+1)^{-2+\varepsilon}\right)\right)\frac{d-1+\din}{T(1+\din x)}\left(1+O\left(\frac1T\right)\right)\\
+x\theta_3\left(1-\frac{a(d+\din)}{T(1+\din x)+\din x+A_{in}-\din}\right)+ax\theta_4\frac{d-1+\din}{T(1+\din x)+\din x+A_{in}-\din}\\
+(1-x)\theta_1\left(1-\frac{d+\din}{T(1+\din x)+\din x+A_{in}}\right)+(1-x)\theta_2\frac{d-1+\din}{T(1+\din x)+\din x+A_{in}},
\end{multline}
we have dropped indicators $[N>0]$ and $[N<T+1]$ because everything in the right-hand side is defined at $x=0$ and $x=1$ too
and corresponding terms are zero anyway due to factors $x$ and $1-x$.

Expand \eqref{recur2}:
\begin{multline*}
E_d(T+1,N)=Tf_d(x)+[d=0]ax+[d=1](1-a)x\\-xf_d(x)\frac{a(d+\din)}{1+\din x}
+axf_{d-1}(x)\frac{d-1+\din}{1+\din x}
-(1-x)f_d(x)\frac{d+\din}{1+\din x}
+(1-x)f_{d-1}(x)\frac{d-1+\din}{1+\din x}\\
+O\left(\frac{(1-x)(d+1)^{-1+\varepsilon}}T\right)+O\left(\frac{ax(d+1)^{-1+\varepsilon}}T\right)\\
+x\theta_3\left(1-\frac{a(d+\din)}{T(1+\din x)+\din x+A_{in}-\din}\right)+ax\theta_4\frac{d-1+\din}{T(1+\din x)+\din x+A_{in}-\din}\\
+(1-x)\theta_1\left(1-\frac{d+\din}{T(1+\din x)+\din x+A_{in}}\right)+(1-x)\theta_2\frac{d-1+\din}{T(1+\din x)+\din x+A_{in}},
\end{multline*}
The sum of terms without $\theta_i$ and $O(\cdot)$ equals $(T+1)f_d(x)$ due to \eqref{fdrecur}.
Denote the sum of other terms as $\theta$. Then
$$
E_d(T+1,N)=(T+1)f_d(x)+\theta,
$$
so we need to prove that $|\theta|\le C(d,x)$.

The expression $x\left(1-\frac{a(d+\din)}{T(1+\din x)+\din x+A_{in}-\din}\right)=\frac{N}{T+1}\left(1-\frac{a(d+\din)}{T+\din N+t_0+\din n_0-\din}\right)$
is always non-negative (provided that $d\le T+t_0$). The expression $1-\frac{d+\din}{T(1+\din x)+\din x+A_{in}}=1-\frac{d+\din}{T+\din N+t_0+\din n_0}$
is also always non-negative. Therefore,
\begin{multline*}
|\theta|\le xC\left(d,x-\frac{1-x}T\right)+(1-x)C\left(d,x+\frac xT\right)\\
-\frac{ax}{T(1+\din x)+\din x+A_{in}-\din}\left((d+\din)C\left(d,x-\frac{1-x}T\right)-[d\ge1](d-1+\din)C\left(d-1,x-\frac{1-x}T\right)\right)\\
-\frac{1-x}{T(1+\din x)+\din x+A_{in}}\left((d+\din)C\left(d,x+\frac xT\right)-[d\ge1](d-1+\din)C\left(d-1,x+\frac xT\right)\right)\\
+O\left(\frac{1-x+ax}{(d+1)^{1-\varepsilon}T}\right).
\end{multline*}

We have $|C''(d,x)|=C_0\left|-\frac{2(d+\din+1)^{1-\varepsilon}}{(1+(d+\din+1)^{1-\varepsilon}x^2)^2}+\frac{8(d+\din+1)^{2-2\varepsilon}x^2}{(1+(d+\din+1)^{1-\varepsilon}x^2)^3}\right|
\le C_0\frac{6(d+\din+1)^{1-\varepsilon}}{(1+(d+\din+1)^{1-\varepsilon}x^2)^2}\le 6(d+\din+1)^{1-\varepsilon}C(d,x)$. Using the Taylor formula, we obtain
\begin{multline*}\left|xC\left(d,x-\frac{1-x}T\right)+(1-x)C\left(d,x+\frac xT\right)-C(d,x)\right|=
\left|\frac{x(1-x)^2}{2T^2}C''(d,\xi_1)+\frac{(1-x)x^2}{2T^2}C''(d,\xi_2)\right|\\
\le\frac{3x(1-x)(d+\din+1)^{1-\varepsilon}C\left(d,x-\frac{1-x}T\right)}{T^2}.\end{multline*}

If $d\ge1$, then
\begin{multline}\label{l1.temp}(d+\din)C(d,x)-(d-1+\din)C(d-1,x)=C_0\int_{d-1}^d\left(\frac{z+\din}{1+(z+\din+1)^{1-\varepsilon}x^2}\right)'_zdz\\
=C_0\int_{d-1}^d\frac{1+\varepsilon(z+\din+1)^{1-\varepsilon}x^2+(1-\varepsilon)(z+\din+1)^{-\varepsilon}x^2}{(1+(z+\din+1)^{1-\varepsilon}x^2)^2}dz
\ge C_0\frac\varepsilon{1+(d+\din+1)^{1-\varepsilon}x^2}=\varepsilon C(d,x).
\end{multline}
Thus, for $d\ge1$
\begin{multline*}
|\theta|\le C(d,x)+\frac{3x(1-x)(d+\din+1)^{1-\varepsilon}C\left(d,x-\frac{1-x}T\right)}{T^2}\\
-\frac{ax\varepsilon C\left(d,x-\frac{1-x}T\right)}{T(1+\din x)+\din x+A_{in}-\din}
-\frac{(1-x)\varepsilon C\left(d,x+\frac xT\right)}{T(1+\din x)+\din x+A_{in}}
+O\left(\frac{1-x+ax}{(d+1)^{1-\varepsilon}T}\right).
\end{multline*}
Since $T\ge T_0$ and $d\le T+t_0$, we have $(d+\din+1)^{1-\varepsilon}\le T\frac\varepsilon{12}$,
$$
\frac{C\left(d,x-\frac{1-x}T\right)}{C\left(d,x+\frac xT\right)}=
\frac{1+(d+\din+1)^{1-\varepsilon}\left(x+\frac xT\right)^2}{1+(d+\din+1)^{1-\varepsilon}\left(x-\frac{1-x}T\right)^2}
\le 1+(d+\din+1)^{1-\varepsilon}\frac{2\left(x+\frac xT\right)}T\le 2,
$$
so $\frac{3x(1-x)(d+\din+1)^{1-\varepsilon}C\left(d,x-\frac{1-x}T\right)}{T^2}\le\frac{\varepsilon(1-x)C\left(d,x+\frac xT\right)}{2T(1+2\din+A_{in})}
\le\frac{\varepsilon(1-x)C\left(d,x+\frac xT\right)}{2(T(1+\din x)+\din x+A_{in})}$. Since $C(d,x)\ge\frac{C_0}{2(d+\din+1)^{1-\varepsilon}}$,
$$
|\theta|\le C(d,x)-\frac{(1-x+ax)\varepsilon C_0}{4(d+\din+1)^{1-\varepsilon}(T(1+\din x)+\din x+A_{in})}
+O\left(\frac{1-x+ax}{(d+1)^{1-\varepsilon}T}\right).
$$
Take $C_0$ such that $\frac{(1-x+ax)\varepsilon C_0}{4(d+\din+1)^{1-\varepsilon}(T(1+\din x)+\din x+A_{in})}\ge
O\left(\frac{1-x+ax}{(d+1)^{1-\varepsilon}T}\right)$, then \eqref{lemm1main} follows.

If $d=0$ and $\din=0$, \eqref{lemm1main} was already proved. The case $d=0$ and $\din>0$
is similar to the case $d\ge1$ with the factor $\varepsilon$ from \eqref{l1.temp} replaced with $\din$.

\end{proof}


\begin{lemma}\label{mainbinom} Let $n$ be a random variable that takes the value $N$ with
probability $\Pr(n=N)=\Pr(\#G_T=N+n_0)$.
Let $f(x)\in C^2[0,1]$. Then
$$
\E f\left(\frac nT\right)=f(\alpha+\gamma)+O\left(\frac{\max_{0\le x\le 1}|f''(x)|}T\right).
$$
\end{lemma}
\begin{proof}
Obviously, $\E1=1$ and $\E n=(\alpha+\gamma)T$. Furthermore,
$$
\E(n(n-1))=\sum_{n=0}^T n(n-1)\binom Tn(\alpha+\gamma)^n\beta^{T-n}=\sum_{n=2}^T T(T-1)\binom T{n-2}(\alpha+\gamma)^n\beta^{T-n}=
T(T-1)(\alpha+\gamma)^2.
$$
Hence,
$$
\E\left(\frac nT\right)^2=(\alpha+\gamma)^2\left(1-\frac1T\right)+\frac{\alpha+\gamma}T.
$$
For any $x\in[0,1]$, we have $f(x)=f(\alpha+\gamma)+(x-\alpha-\gamma)f'(\alpha+\gamma)+\frac{(x-\alpha-\gamma)^2}2f''(\xi)$,
where $\xi$ is some point between $x$ and $\alpha+\gamma$. Therefore,
$$
\left|\E f\left(\frac nT\right)-f(\alpha+\gamma)\right|=
\left|\E\left(\frac12\left(\frac nT-\alpha-\gamma\right)^2f''(\xi)\right)\right|\le\frac{(\alpha+\gamma)-(\alpha+\gamma)^2}{2T}\max_{0\le x\le1}|f''(x)|.
$$
\end{proof}

Now theorem in the case $\gamma=0$ follows from Lemma \ref{expect1} and Lemma \ref{mainbinom}. However,
in the case $\gamma>0$ we need slightly more subtle approach to estimate the remainder term.
Azuma--Hoeffding inequality implies that (in terms of Lemma \ref{mainbinom}) $\Pr(n<(\alpha+\gamma-\lambda)T)\le\exp(-T\lambda^2/2)$
for $\lambda>0$, so $\E\frac1{1+(d+\din+1)^{1-\varepsilon}n^2/T^2}=\E\frac{\left[\frac nT\le\frac{\alpha+\gamma}2\right]}{1+(d+\din+1)^{1-\varepsilon}n^2/T^2}
+\E\frac{\left[\frac nT\ge\frac{\alpha+\gamma}2\right]}{1+(d+\din+1)^{1-\varepsilon}n^2/T^2}\le\exp\left(-T\frac{(\alpha+\gamma)^2}8\right)+
\frac1{1+(d+\din+1)^{1-\varepsilon}\frac{(\alpha+\gamma)^2}4}=O(d^{-1+\varepsilon})$ if $d\le T+t_0$. For $d>T+t_0$ the theorem holds
due to $n_{in}=0$.
\end{proof}
\section{Concentration for number of vertices}
Let $D_{d_1,d_2}(T,N)=\E\left(n_{in}(T+t_0,d_1)n_{in}(T+t_0,d_2)|\#G=N+n_0\right)-
[d_1=d_2]E_{d_1}(T,N)$.
\begin{multline*}
D_{d_1,d_2}(T,N)=\sum_{i,j=1}^{N+n_0}\Pr(\deg_{in,T}(i)=d_1,\deg_{in,T}(j)=d_2|\#G_T=N+n_0)\\-[d_1=d_2]E_{d_1}(T,N)
=\sum_{\substack{i,j=1\\i\ne j}}^{N+n_0}\Pr(\deg_{in,T}(i)=d_1,\deg_{in,T}(j)=d_2|\#G_T=N+n_0).
\end{multline*}
\begin{lemma}\label{dbound}
If $\gamma+\din>0$, then
\begin{multline*}D_{d_1,d_2}(T,N)=T^2f_{d_1}\left(\frac NT\right)f_{d_2}\left(\frac NT\right)+O\Bigg(
\frac{T(d_1+\din+1)^{-1-\frac1{\cin(N/T)+\varepsilon}}}{1+(d_2+\din+1)^{1-\varepsilon}(N/T)^2}+
\frac{T(d_2+\din+1)^{-1-\frac1{\cin(N/T)+\varepsilon}}}{1+(d_1+\din+1)^{1-\varepsilon}(N/T)^2}\\
+\frac1{(1+(d_1+\din+1)^{1-\varepsilon}(N/T)^2)(1+(d_2+\din+1)^{1-\varepsilon}(N/T)^2)}
\Bigg).\end{multline*}
\end{lemma}
\begin{proof}
The recurrent equation for $D_{d_1,d_2}(T,N)$ is obtained similarly to \eqref{recur1}: for $i,j\in G_T$, $i\ne j$,
\begin{multline*}
\Pr(\deg_{in,T+1}(i)=d_1,\deg_{in,T+1}(j)=d_2,\#G_{T+1}=N+n_0,\dagger)\\=
\alpha \Pr(\deg_{in,T}(i)=d_1,\deg_{in,T}(j)=d_2,\#G_T=N+n_0-1)\left(1-\frac{d_1+d_2+2\din}{T+\din N+A_{in}-\din}\right)\\
+\alpha \Pr(\deg_{in,T}(i)=d_1-1,\deg_{in,T}(j)=d_2,\#G_T=N+n_0-1)\frac{d_1-1+\din}{T+\din N+A_{in}-\din}\\
+\alpha \Pr(\deg_{in,T}(i)=d_1,\deg_{in,T}(j)=d_2-1,\#G_T=N+n_0-1)\frac{d_2-1+\din}{T+\din N+A_{in}-\din},
\end{multline*}
\begin{multline*}
\Pr(\deg_{in,T+1}(i)=d_1,\deg_{in,T+1}(j)=d_2,\#G_{T+1}=N+n_0,\dagger\dagger)\\=
\beta \Pr(\deg_{in,T}(i)=d_1,\deg_{in,T}(j)=d_2,\#G_T=N+n_0)\left(1-\frac{d_1+d_2+2\din}{T+\din N+A_{in}}\right)\\
+\beta \Pr(\deg_{in,T}(i)=d_1-1,\deg_{in,T}(j)=d_2,\#G_T=N+n_0)\frac{d_1-1+\din}{T+\din N+A_{in}}\\
+\beta \Pr(\deg_{in,T}(i)=d_1,\deg_{in,T}(j)=d_2-1,\#G_T=N+n_0)\frac{d_2-1+\din}{T+\din N+A_{in}},
\end{multline*}
\begin{multline*}
\Pr(\deg_{in,T+1}(i)=d_1,\deg_{in,T+1}(j)=d_2,\#G_{T+1}=N+n_0,\ddagger)\\=
\gamma \Pr(\deg_{in,T}(i)=d_1,\deg_{in,T+1}(j)=d_2,\#G_T=N+n_0-1);
\end{multline*}
\begin{multline}\label{drecur}
D_{d_1,d_2}(T+1,N)\\=
[N>0,d_1=0]\alpha E_{d_2}(T,N-1)\frac{N}{(\alpha+\gamma)(T+1)}+[N>0,d_2=0]\alpha E_{d_1}(T,N-1)\frac{N}{(\alpha+\gamma)(T+1)}\\
+[N>0]\alpha D_{d_1,d_2}(T,N-1)\left(1-\frac{d_1+d_2+2\din}{T+\din N+A_{in}-\din}\right)\frac{N}{(\alpha+\gamma)(T+1)}\\
+[N>0]\alpha D_{d_1-1,d_2}(T,N-1)\frac{d_1-1+\din}{T+\din N+A_{in}-\din}\frac{N}{(\alpha+\gamma)(T+1)}\\
+[N>0]\alpha D_{d_1,d_2-1}(T,N-1)\frac{d_2-1+\din}{T+\din N+A_{in}-\din}\frac{N}{(\alpha+\gamma)(T+1)}\\
+[N<T+1]\beta D_{d_1,d_2}(T,N)\left(1-\frac{d_1+d_2+2\din}{T+\din N+A_{in}}\right)\frac{T+1-N}{\beta(T+1)}\\
+[N<T+1]\beta D_{d_1-1,d_2}(T,N)\frac{d_1-1+\din}{T+\din N+A_{in}}\frac{T+1-N}{\beta(T+1)}\\
+[N<T+1]\beta D_{d_1,d_2-1}(T,N)\frac{d_2-1+\din}{T+\din N+A_{in}}\frac{T+1-N}{\beta(T+1)}\\
+[N>0,d_1=1]\gamma E_{d_2}(T,N-1)\frac{N}{(\alpha+\gamma)(T+1)}+[N>0,d_2=1]\gamma E_{d_1}(T,N-1)\frac{N}{(\alpha+\gamma)(T+1)}\\
+[N>0]\gamma D_{d_1,d_2}(T,N-1)\frac{N}{(\alpha+\gamma)(T+1)}.
\end{multline}

We use induction by $T$ and, for fixed $T$, by $d_1+d_2$ to prove the formula
\begin{equation}\label{conmain}
\left|D_{d_1,d_2}(T,N)-T^2f_{d_1}\left(\frac NT\right)f_{d_2}\left(\frac NT\right)\right|\le C\left(T,d_1,d_2,\frac NT\right),
\end{equation}
$C(T,d_1,d_2,x)=TC_1(d_1,d_2,x)+TC_2(d_1,d_2,x)+C_3(d_1,d_2,x)$,
\begin{eqnarray*}
C_1(d_1,d_2,x)&=&C_0\frac{\hat f_{d_1}(x)}{1+(d_2+\din+1)^{1-\varepsilon}x^2},\\
C_2(d_1,d_2,x)&=&C_0\frac{\hat f_{d_2}(x)}{1+(d_1+\din+1)^{1-\varepsilon}x^2},\\
C_3(d_1,d_2,x)&=&\frac{C_0[d_1\ge1,d_2\ge1]}{\left(1+(d_1+\din+1)^{1-\varepsilon}x^2\right)\left(1+(d_2+\din+1)^{1-\varepsilon}x^2\right)},\\
\hat f_0(x)&=&1,\\
\hat f_d(x)&=&\prod_{i=2}^d\frac{i+\din-1}{i+\din+\frac1{\cin+\varepsilon}}\mbox{ for }d\ge1.
\end{eqnarray*}
Since $\frac{i+\din-1}{i+\din+\frac1{\cin+\varepsilon}}=1-\left(1+\frac1{\cin+\varepsilon}\right)\frac1{i+\din+\frac1{\cin+\varepsilon}}
\le\left(1+\frac1{i+\din+\frac1{\cin+\varepsilon}}\right)^{-1-\frac1{\cin+\varepsilon}}$ and\\
$\frac{i+\din-1}{i+\din+\frac1{\cin+\varepsilon}}=\left(1+\left(1+\frac1{\cin+\varepsilon}\right)\frac1{i+\din-1}\right)^{-1}
\ge\left(1+\frac1{i+\din-1}\right)^{-1-\frac1{\cin+\varepsilon}}$, we have
$$(1+\din)^{1+\frac1{1+\varepsilon}}(d+\din)^{-1-\frac1{\cin+\varepsilon}}\le\hat f_d(x)
\le\left(2+\din+\frac1{\varepsilon}\right)^{1+\frac1{\varepsilon}}\left(d+\din+1+\frac1{1+\varepsilon}\right)^{-1-\frac1{\cin+\varepsilon}},$$
so $\hat f_d(x)=\Theta\left((d+\din+1)^{-1-\frac1{\cin+\varepsilon}}\right)$, where both implied constants depend only on $\varepsilon$
and parameters of the model. For same reasons as for $f_d(x)$, we have $\frac{\hat f'_d(x)}{\hat f_d(x)}=O\left(\ln(d+\din+1)\right)$
and $\frac{\hat f''_d(x)}{\hat f_d(x)}=O\left(\ln^2(d+\din+1)\right)$.

Since $\left|\frac d{dx}\left(\frac1{1+\lambda x^2}\right)\right|=\frac{2\lambda x}{(1+\lambda x^2)^2}
\le\frac{\lambda x}{\sqrt{\lambda x^2}(1+\lambda x^2)}=\sqrt\lambda\frac1{1+\lambda x^2}$
and $\left|\frac{d^2}{dx^2}\left(\frac1{1+\lambda x^2}\right)\right|\le 6\lambda\frac1{1+\lambda x^2}$,
we have\\ $\frac{C_i'(d_1,d_2,x)}{C_i(d_1,d_2,x)}=O\left((d_1+\din+1)^{\frac{1-\varepsilon}2}+(d_2+\din+1)^{\frac{1-\varepsilon}2}\right)$
and\\
$\frac{C_i''(d_1,d_2,x)}{C_i(d_1,d_2,x)}=O\left((d_1+\din+1)^{1-\varepsilon}+(d_2+\din+1)^{1-\varepsilon}\right)$.
If $\xi\in\left[x-\frac{1-x}T,x+\frac xT\right]$ and $d\le T+t_0$, we have
\begin{multline*}
\hat f_d'(\xi)=O\left((d+1+\din)^{-1-\frac1{\cin(\xi)+\varepsilon}}\ln(d+1+\din)\right)=O\left((d+1+\din)^{-1-\frac1{\cin(x)+\varepsilon}+O\left(\frac1T\right)}\ln(d+1+\din)\right)\\
=O\left(\hat f_d(x)\ln(d+1+\din)e^{O\left(\frac{\ln(d+1+\din)}T\right)}\right)=O\left(\hat f_d(x)\ln(d+1+\din)\right),
\end{multline*}
$\max\left\{\left|\hat f_d\left(x-\frac{1-x}T\right)-\hat f_d(x)\right|,\left|\hat f_d\left(x+\frac{x}T\right)-\hat f_d(x)\right|\right\}
\le\frac1T\max_{\xi\in\left[x-\frac{1-x}T,x+\frac xT\right]}\left|\hat f'_d(\xi)\right|=\hat f_d(x)O\left(\frac{\ln(d+1+\din)}T\right)$,\\
$\hat f_d\left(x-\frac{1-x}T\right)=\hat f_d(x)\left(1+O\left(\frac{\ln(d+1+\din)}T\right)\right)$
and $\hat f_d\left(x+\frac{x}T\right)=\hat f_d(x)\left(1+O\left(\frac{\ln(d+1+\din)}T\right)\right)$.
Moreover,\\ $\left|\frac{1+\lambda x^2}{1+\lambda\xi^2}-1\right|\le\frac{2\lambda|\xi-x|}{1+\lambda\xi^2}\le\sqrt\lambda|\xi-x|$,
so $C_i\left(d_1,d_2,x+\frac{x}T\right)=C_i(d_1,d_2,x)\left(1+O\left(\frac{(d_1+1+\din)^{\frac{1-\varepsilon}2}+(d_2+1+\din)^{\frac{1-\varepsilon}2}}T\right)\right)$
and $C_i\left(d_1,d_2,x-\frac{1-x}T\right)=C_i(d_1,d_2,x)\left(1+O\left(\frac{(d_1+1+\din)^{\frac{1-\varepsilon}2}+(d_2+1+\din)^{\frac{1-\varepsilon}2}}T\right)\right)$.

Select $T_0$ such that for $T\ge T_0$ and $d_1+d_2\le T+t_0$ the following inequalities hold:
\begin{eqnarray}
\left|C_i\left(d_1,d_2,x-\frac{1-x}T\right)-C_i(d_1,d_2,x)\right|&\le& \frac\varepsilon3C_i(d_1,d_2,x),\label{l2diff1a}\\
\left|C_i\left(d_1,d_2,x+\frac{x}T\right)-C_i(d_1,d_2,x)\right|&\le& \frac\varepsilon3C_i(d_1,d_2,x),\label{l2diff1b}\\
\left|C_i''(d_1,d_2,x)\right|&\le&\frac{\varepsilon}{3(1+2\din+A_{in})}TC_i(d_1,d_2,x).\label{l2diff2}
\end{eqnarray}
(Note that validity of these inequalities does not depend on $C_0$, so $T_0$ does not depend on $C_0$, which has not been selected yet.)

There are only finite number of pairs $(T,N)$ with $T\le T_0$, so it is possible to select $C_0$
such that \eqref{conmain} holds for $T\le T_0$. This is the base of induction.

Before proceeding with induction, we need to establish \eqref{conmain} for the cases
when $1-\frac{d_1+d_2+2\din}{T-1+\din N+t_0+\din n_0-\din}<0$.
If $d_1+d_2>T+t_0$, then $D_{d_1,d_2}(T,N)=0$, so \eqref{conmain} holds. If $N=O(1)$, then the bound $D_{d_1,d_2}\le (N+n_0)^2=O(1)$
is sufficient. The only other case when the expression above can be negative is $d_1+d_2=T+t_0$ and $\din=0$.
Then $\gamma>0$ and $a<1$; since $d_1+d_2=T+t_0$ means that $(\ddagger)$ has been used at most once not counting the first step,
the probability of that given $\#G_T=N+n_0$ is not greater than $Na^{N-2}$, so $D_{d_1,d_2}(T,N)\le(N+n_0)Na^{N-2}=O\left(\frac1{1+N}\right)$,
which is less than $C(T,d_1,d_2,x)$ if $C_0$ is sufficiently large.

Now assume that \eqref{conmain} is proved for some value of $T\ge T_0$ and consider the value $T+1$.
Assume also that $1-\frac{d_1+d_2+2\din}{T+\din N+t_0+\din n_0-\din}\ge0$.
Let $0\le N\le T+1$ and $x=\frac N{T+1}$.

If $\xi\in\left[x-\frac{1-x}T,x+\frac xT\right]$ and $d\le T+t_0$, we have
\begin{multline*}
f_d'(\xi)=O\left((d+1+\din)^{-1-\frac1{\cin(\xi)+\varepsilon}}\right)=O\left((d+1+\din)^{-1-\frac1{\cin(x)+\varepsilon}+O\left(\frac1T\right)}\right)\\
=O\left((d+1+\din)^{-1-\frac1{\cin(x)+\varepsilon}}\right)e^{O\left(\frac{\ln(d+1+\din)}T\right)}=O\left((d+1+\din)^{-1-\frac1{\cin(x)+\varepsilon}}\right)
\end{multline*}
and $f_d''(\xi)=O\left((d+1+\din)^{-1-\frac1{\cin(x)+\varepsilon}}\right)$ for the same reason.

If $N\ne T+1$, the inductive hypothesis and the Taylor formula imply that
\begin{multline*}
D_{d_1,d_2}(T,N)=T^2f_{d_1}\left(x+\frac xT\right)f_{d_2}\left(x+\frac xT\right)+\theta_1\\=
T^2f_{d_1}(x)f_{d_2}(x)+xT(f_{d_1}(x)f_{d_2}(x))'+O\left((d_1+1+\din)^{-1-\frac1{\cin+\varepsilon}}(d_2+1+\din)^{-1-\frac1{\cin+\varepsilon}})\right)+\theta_1,
\end{multline*}
$$
D_{d_1-1,d_2}(T,N)=T^2f_{d_1-1}(x)f_{d_2}(x)+O\left(T(d_1+1+\din)^{-1-\frac1{\cin+\varepsilon}}(d_2+1+\din)^{-1-\frac1{\cin+\varepsilon}})\right)+\theta_2,
$$
$$
D_{d_1,d_2-1}(T,N)=T^2f_{d_1}(x)f_{d_2-1}(x)+O\left(T(d_1+1+\din)^{-1-\frac1{\cin+\varepsilon}}(d_2+1+\din)^{-1-\frac1{\cin+\varepsilon}})\right)+\theta_3,
$$
where $|\theta_1|\le C\left(T,d_1,d_2,x+\frac xT\right)$, $|\theta_2|\le C\left(T,d_1-1,d_2,x+\frac xT\right)$ if $d_1\ge1$
and $\theta_2=0$ if $d_1=0$,\\ $|\theta_3|\le C\left(T,d_1,d_2-1,x+\frac xT\right)$ if $d_2\ge1$ and $\theta_3=0$ if $d_2=0$.
If $N\ne0$, similarly
\begin{multline*}
D_{d_1,d_2}(T,N-1)=T^2f_{d_1}\left(x-\frac{1-x}T\right)f_{d_2}\left(x-\frac{1-x}T\right)+\theta_4\\=
T^2f_{d_1}(x)f_{d_2}(x)-(1-x)T(f_{d_1}(x)f_{d_2}(x))'+O\left((d_1+1+\din)^{-1-\frac1{\cin+\varepsilon}}(d_2+1+\din)^{-1-\frac1{\cin+\varepsilon}})\right)+\theta_4,
\end{multline*}
$$
D_{d_1-1,d_2}(T,N-1)=T^2f_{d_1-1}(x)f_{d_2}(x)+O\left(T(d_1+1+\din)^{-1-\frac1{\cin+\varepsilon}}(d_2+1+\din)^{-1-\frac1{\cin+\varepsilon}})\right)+\theta_5,
$$
$$
D_{d_1,d_2-1}(T,N-1)=T^2f_{d_1}(x)f_{d_2-1}(x)+O\left(T(d_1+1+\din)^{-1-\frac1{\cin+\varepsilon}}(d_2+1+\din)^{-1-\frac1{\cin+\varepsilon}})\right)+\theta_6,
$$
where $|\theta_4|\le C\left(T,d_1,d_2,x-\frac{1-x}T\right)$, $|\theta_5|\le C\left(T,d_1-1,d_2,x-\frac{1-x}T\right)$ if $d_1\ge1$
and $\theta_5=0$ if $d_1=0$,\\ $|\theta_6|\le C\left(T,d_1,d_2-1,x-\frac{1-x}T\right)$ if $d_2\ge1$ and $\theta_6=0$ if $d_2=0$.

Substitute these representations and results of Lemma \ref{expect1} to \eqref{drecur} and expand:
\begin{multline*}
D_{d_1,d_2}(T+1,N)=T^2f_{d_1}(x)f_{d_2}(x)
-(1-x+ax)\frac{d_1+d_2+2\din}{1+\din x}Tf_{d_1}(x)f_{d_2}(x)\\
+(1-x+ax)\frac{d_1-1+\din}{1+\din x}Tf_{d_1-1}(x)f_{d_2}(x)
+(1-x+ax)\frac{d_2-1+\din}{1+\din x}Tf_{d_1}(x)f_{d_2-1}(x)\\
+[d_1=0]axTf_{d_2}(x)
+[d_2=0]axTf_{d_1}(x)
+[d_1=1](1-a)xTf_{d_2}(x)
+[d_2=1](1-a)xTf_{d_1}(x)\\
+O\left((d_1+d_2+2)(d_1+1+\din)^{-1-\frac1{\cin+\varepsilon}}(d_2+1+\din)^{-1-\frac1{\cin+\varepsilon}}\right)\\
+O\left(\frac{ax[d_1=0]+(1-a)x[d_1=1]}{1+(d_2+\din+1)^{1-\varepsilon}x^2}\right)
+O\left(\frac{ax[d_2=0]+(1-a)x[d_2=1]}{1+(d_1+\din+1)^{1-\varepsilon}x^2}\right)\\
+\theta_1(1-x)\left(1-\frac{d_1+d_2+2\din}{T(1+\din x)+\din x+A_{in}}\right)\\
+\theta_2(1-x)\frac{d_1-1+\din}{T(1+\din x)+\din x+A_{in}}
+\theta_3(1-x)\frac{d_2-1+\din}{T(1+\din x)+\din x+A_{in}}\\
+\theta_4 x\left(1-\frac{a(d_1+d_2+2\din)}{T(1+\din x)+\din x+A_{in}-\din}\right)\\
+\theta_5 ax\frac{d_1-1+\din}{T(1+\din x)+\din x+A_{in}-\din}
+\theta_6 ax\frac{d_2-1+\din}{T(1+\din x)+\din x+A_{in}-\din}.
\end{multline*}
The sum of terms without $\theta_i$ and $O(\cdot)$ equals $T^2f_{d_1}(x)f_{d_2}(x)+2Tf_{d_1}(x)f_{d_2}(x)$ due to \eqref{fdrecur}.
Denote the sum of other terms as $\theta+f_{d_1}(x)f_{d_2}(x)$. Then $D_{d_1,d_2}(T+1,N)=(T+1)^2f_{d_1}(x)f_{d_2}(x)+\theta$,
so we need to prove that $|\theta|\le (T+1)(C_1(d_1,d_2,x)+C_2(d_1,d_2,x))+C_3(d_1,d_2,x)$.

From the Taylor formula and \eqref{l2diff2} it follows that
$$
\left|xC\left(T,d_1,d_2,x-\frac{1-x}T\right)+(1-x)C\left(T,d_1,d_2,x+\frac xT\right)-C(T,d_1,d_2,x)\right|\le\frac{\varepsilon x(1-x)C(T,d_1,d_2,x)}{6T(1+2\din+A_{in})}.
$$
We know from \eqref{l1.temp} that
$(d_1+\din)C_3(d_1,d_2,x)-(d_1-1+\din)C_3(d_1-1,d_2,x)\ge\varepsilon C_3(d_1,d_2,x)$ if $d_1\ge2$
and the same holds for $d_1=0$ and $d_1=1$ due to $C_3(0,d_2,x)=0$.
Similarly, $(d+\din)(d+1+\din)^{-1+\varepsilon}-(d-1+\din)(d+\din)^{-1+\varepsilon}=\int_{d+\din}^{d+1+\din}
\left((z-1)z^{-1+\varepsilon}\right)'_zdz\ge0$ if $d\ge1$,
so $(d_1+\din)C_2(d_1,d_2,x)-(d_1-1+\din)C_2(d_1-1,d_2,x)\ge0$ (even for $d_1=0$).
It follows from the definition of $\hat f_d(x)$ that $(d_1+\din)C_1(d_1,d_2,x)-(d_1-1+\din)C_1(d_1-1,d_2,x)\ge-\frac1{\cin+\varepsilon}C_1(d_1,d_2,x)$
(with the equality if $d_1\ge2$).
Therefore,
$$
(d_1+\din)C(T,d_1,d_2,x)-(d_1+\din-1)C(T,d_1-1,d_2,x)\ge -\frac{T}{\cin+\varepsilon}C_1(d_1,d_2,x)+\varepsilon C_3(d_1,d_2,x),
$$
$$
(d_2+\din)C(T,d_1,d_2,x)-(d_2+\din-1)C(T,d_1,d_2-1,x)\ge -\frac{T}{\cin+\varepsilon}C_2(d_1,d_2,x)+\varepsilon C_3(d_1,d_2,x),
$$
\begin{multline*}
(d_1+d_2+2\din)C(T,d_1,d_2,x)-(d_1+\din-1)C(T,d_1-1,d_2,x)-(d_2+\din-1)C(T,d_1,d_2-1,x)\\
\ge -\frac{T}{\cin+\varepsilon}(C_1(d_1,d_2,x)+C_2(d_1,d_2,x))+2\varepsilon C_3(d_1,d_2,x).
\end{multline*}
Thus,
\begin{multline*}
|\theta|\le C(T,d_1,d_2,x)+\frac{\varepsilon x(1-x)C(T,d_1,d_2,x)}{6T(1+2\din+A_{in})}\\
+\frac{T(1-x)\frac1{\cin+\varepsilon}\left(C_1\left(d_1,d_2,x+\frac xT\right)+C_2\left(d_1,d_2,x+\frac xT\right)\right)-(1-x)2\varepsilon C_3\left(d_1,d_2,x+\frac xT\right)}{T(1+\din x)+\din x+A_{in}}\\
+\frac{Tax\frac1{\cin+\varepsilon}\left(C_1\left(d_1,d_2,x-\frac{1-x}T\right)+C_2\left(d_1,d_2,x-\frac{1-x}T\right)\right)-ax2\varepsilon C_3\left(d_1,d_2,x-\frac{1-x}T\right)}{T(1+\din x)+\din x+A_{in}-\din}\\
+O\left(\frac{(d_1+1+\din)^{-1-\frac1{\cin+\varepsilon}}}{1+(d_2+\din+1)^{1-\varepsilon}x^2}+\frac{(d_2+1+\din)^{-1-\frac1{\cin+\varepsilon}}}{1+(d_1+\din+1)^{1-\varepsilon}x^2}\right).
\end{multline*}
Apply bounds \eqref{l2diff1a} and \eqref{l2diff1b}. Terms without $O(\cdot)$ become a linear combination
of $C_i(d_1,d_2,x)$. $C_1$ and $C_2$ have the coefficient that is not greater than
$T+\frac{\varepsilon x(1-x)}{6(1+2\din+A_{in})}+\frac{(1-x+ax)\left(1+\frac\varepsilon3\right)}{(1+\din x)(\cin+\varepsilon)}\le
T+\frac\varepsilon6+\frac{\cin\left(1+\frac\varepsilon3\right)}{\cin+\varepsilon}\le
T+\frac\varepsilon6+\frac{1+\frac\varepsilon3}{1+\varepsilon}\le T+1-\frac\varepsilon6$.
$C_3$ has the coefficient that is not greater than
$1+\frac{\varepsilon x(1-x)}{6T(1+2\din+A_{in})}-\frac{(1-x+ax)2\varepsilon\left(1-\frac\varepsilon3\right)^2}{T(1+2\din+A_{in})}\le
1+\frac{\varepsilon(1-x)}{T(1+2\din+A_{in})}\left(\frac16-2\left(1-\frac\varepsilon3\right)^2\right)\le1$.
It remains to select $C_0$ such that $\frac\varepsilon6\left(C_1(d_1,d_2,x)+C_2(d_1,d_2,x)\right)\ge
O\left(\frac{(d_1+1+\din)^{-1-\frac1{\cin+\varepsilon}}}{1+(d_2+\din+1)^{1-\varepsilon}x^2}+\frac{(d_2+1+\din)^{-1-\frac1{\cin+\varepsilon}}}{1+(d_1+\din+1)^{1-\varepsilon}x^2}\right)$.
\end{proof}

\begin{proof}[Proof of Theorem 2]
Without loss of generality assume $\varepsilon<1$.
Let $x=\frac NT$.

We have $\Pr\left(|x-(\alpha+\gamma)|>\sqrt{\frac{\ln T}T}\right)\to0$ as $T\to\infty$;
therefore, it is sufficient to show that
$$\Pr\left(|n_{in}(t,d)-\of d t|>\left(\sqrt{\of d t}+(d+1)^{-\frac12+\varepsilon}\right)\ln t\Big| \#G_T=N+n_0\right)\to0$$
uniformly if $|x-(\alpha+\gamma)|\le\sqrt{\frac{\ln T}T}$.

From now on, assume $|x-(\alpha+\gamma)|\le\sqrt{\frac{\ln T}T}$. Then $\cin=\Theta(1)$ and similarly to
\eqref{fdbound} we have $f_d(x)=O\left((d+1)^{-1-\frac1\cin}\right)$ and $f_d'(x)=O\left((d+1)^{-1-\frac1\cin}\ln(d+2)\right)$.
Furthermore,\\ $|\cin-\ocin|=O(|x-(\alpha+\gamma)|\max|\cin'|)=O\left(\frac{\sqrt{\ln T}}T\right)$,
$f_d(x)=O\left((d+1)^{-1-\frac1\ocin+O\left(\sqrt{\frac{\ln T}T}\right)}\right)=O\left((d+1)^{-1-\frac1\ocin}\right)$.
Lemma \eqref{expect1} and the equality $\of d=f_d(\alpha+\gamma)$ imply
\begin{multline*}
\E(n_{in}(t,d)|\#G_T=N+n_0)=E_d(T,N)=Tf_d(x)+O\left((d+1)^{-1+\varepsilon}\right)\\
=T\of d+O\left(\sqrt{T\ln T}(d+1)^{-1-\frac1\ocin}\ln(d+2)+(d+1)^{-1+\varepsilon}\right).
\end{multline*}
Lemmas \eqref{expect1} and \eqref{dbound} imply
$$
D(n_{in}(t,d)|\#G_T=N+n_0)=D_{d,d}(T,N)+E_d(T,N)-E_d(T,N)^2=O\left(T(d+1)^{-1-\frac1\ocin}+(d+1)^{-1+\varepsilon}\right).
$$
If $\alpha=0$ and $d=0$, vertices with zero in-degree can come only from $G_0$, so $n_{in}(t,0)=O(1)$ in this case,
$\of0=0$ and the theorem holds. Otherwise, $\of d=\Theta\left((d+1)^{-1-\frac1\ocin}\right)$ and
if $t$ is sufficiently large, $|\E n_{in}(t,d)-\of dt|\le\frac12\left(\sqrt{\of d t}+(d+1)^{-\frac12+\varepsilon}\right)\ln t$,
\begin{multline*}
\Pr\left(|n_{in}(t,d)-\of d t|>\left(\sqrt{\of d t}+(d+1)^{-\frac12+\varepsilon}\right)\ln t\Big| \#G_T=N+n_0\right)\\
\le \Pr\left(|n_{in}(t,d)-\E(n_{in}(t,d)|\#G_T=N+n_0)|>\frac12\left(\sqrt{\of d t}+(d+1)^{-\frac12+\varepsilon}\right)\ln t\Big| \#G_T=N+n_0\right)\\
\le\frac{D(n_{in}(t,d)|\#G_T=N+n_0)}{\frac14\left(\sqrt{\of d t}+(d+1)^{-\frac12+\varepsilon}\right)^2\ln^2 t}
=O\left(\frac{t(d+1)^{-1-\frac1\ocin}+(d+1)^{-1+\varepsilon}}{(t(d+1)^{-1-\frac1\ocin}+(d+1)^{-1+2\varepsilon})\ln^2t}\right)
=O\left(\frac1{\ln^2t}\right)=o(1).
\end{multline*}
\end{proof}

\section{Expected number of edges between vertices with the given degree}
\subsection{Recurrent equation}
Similar to calculation of number of vertices, let
\begin{multline*}E_X(T,N,d_1,d_2)=\E(X(T+t_0,d_1,d_2)|\#G=N+n_0)\\
=\sum_{\substack{i,j=1\\i\ne j}}^{N+n_0}\frac{\E\left([\#G=N+n_0,\deg_{out,T}(i)=d_1,\deg_{in,T}(j)=d_2]V_T(i,j)\right)}{\Pr(\#G=N+n_0)},
\end{multline*}
where $V_T(i,j)$ is the number of edges from $i$ to $j$.

We start with a recurrent equation for $E_X$.
Again, let $G_{T+1}$ be a random graph from $\mathcal G(T+t_0+1)$ with $N+n_0$ vertices.
\begin{itemize}
\item If $G_{T+1}$ is constructed from $G_T\in\mathcal G(T+t_0)$ using $(\dagger)$, then $N>0$,
$G_T$ has $N+n_0-1$ vertices. Let $N+n_0$ be the new vertex and $w\in G_T$ be the target vertex of the new edge.
\begin{itemize}
\item If $i=N+n_0$ and $j=w$, then
$$[\deg_{out,T+1}(i)=d_1,\deg_{in,T+1}(j)=d_2]V_{T+1}(i,j)=[d_1=1,\deg_{in,T}(j)=d_2-1].$$
\item If $i=N+n_0$ and $j\ne w$, then
$$[\deg_{out,T+1}(i)=d_1,\deg_{in,T+1}(j)=d_2]V_{T+1}(i,j)=0.$$
\item If $i\ne N+n_0$ and $j=w$, then
$$[\deg_{out,T+1}(i)=d_1,\deg_{in,T+1}(j)=d_2]V_{T+1}(i,j)=[\deg_{out,T}(i)=d_1,\deg_{in,T}(j)=d_2-1]V_T(i,j).$$
\item If $i\ne N+n_0$ and $j\ne w$, then
$$[\deg_{out,T+1}(i)=d_1,\deg_{in,T+1}(j)=d_2]V_{T+1}(i,j)=[\deg_{out,T}(i)=d_1,\deg_{in,T}(j)=d_2]V_T(i,j).$$
\end{itemize}
Thus,
\begin{multline*}\sum_{\substack{i,j=1\\i\ne j}}^{N+n_0}[\deg_{out,T+1}(i)=d_1,\deg_{in,T+1}(j)=d_2]V_{T+1}(i,j)\\=
\sum_{\substack{i,j=1\\i\ne j}}^{N+n_0-1}[\deg_{out,T}(i)=d_1,\deg_{in,T}(j)=d_2]V_T(i,j)+
[d_1=1,\deg_{in,T}(w)=d_2-1]\\+
\sum_{i=1}^{N+n_0-1}[\deg_{out,T}(i)=d_1,\deg_{in,T}(w)=d_2-1]V_{T}(i,w)\\-
\sum_{i=1}^{N+n_0-1}[\deg_{out,T}(i)=d_1,\deg_{in,T}(w)=d_2]V_{T}(i,w)
\end{multline*}
Recall that $w$ is selected with probability $\frac{\deg_{in}(w)+\din}{T+\din(N-1)+A_{in}}$.
\begin{multline*}
\E\left(X(T+1+t_0,d_1,d_2|G_T,(\dagger)\right)=
X(T+t_0,d_1,d_2)\left(1-\frac{d_2+\din}{T+\din(N-1)+A_{in}}\right)\\+
\frac{d_2-1+\din}{T+\din(N-1)+A_{in}}[d_1=1]n_{in}(t,d_2-1)+\frac{d_2-1+\din}{T+\din(N-1)+A_{in}}X(T+t_0,d_1,d_2-1)
\end{multline*}
\item If $G_{T+1}$ is constructed from $G_T\in\mathcal G(T+t_0)$ using $(\dagger\dagger)$, then
$N<T+1$, $G_T$ has $N+n_0$ vertices. Let $v$ and $w$ be correspondingly the source and the target of the new edge.
Then $\deg_{out,T+1}(i)=\deg_{out,T}(i)-[i=v]$, $\deg_{in,T+1}(j)=\deg_{in,T}(j)-[j=w]$,
$V_{T+1}(i,j)=V_T(i,j)+[i=v,j=w]$, so
\begin{multline*}\sum_{\substack{i,j=1\\i\ne j}}^{N+n_0}[\deg_{out,T+1}(i)=d_1,\deg_{in,T+1}(j)=d_2]V_{T+1}(i,j)=
\sum_{\substack{i,j=1\\i\ne j}}^{N+n_0}[\deg_{out,T}(i)=d_1,\deg_{in,T}(j)=d_2]V_T(i,j)\\+
\sum_{\substack{i=1\\i\ne w}}^{N+n_0}[\deg_{out,T}(i)=d_1,\deg_{in,T}(w)=d_2-1]V_T(i,w)-
\sum_{\substack{i=1\\i\ne w}}^{N+n_0}[\deg_{out,T}(i)=d_1,\deg_{in,T}(w)=d_2]V_T(i,w)\\+
\sum_{\substack{j=1\\j\ne v}}^{N+n_0}[\deg_{out,T}(v)=d_1-1,\deg_{in,T}(j)=d_2]V_T(v,j)-
\sum_{\substack{j=1\\j\ne v}}^{N+n_0}[\deg_{out,T}(v)=d_1,\deg_{in,T}(j)=d_2]V_T(v,j)\\+
[\deg_{out,T}(v)=d_1-1,\deg_{out,T}(w)=d_2-1,v\ne w](V_T(v,w)+1)\\-[\deg_{out,T}(v)=d_1,\deg_{in,T}(w)=d_2-1,v\ne w]V_T(v,w)\\
-[\deg_{out,T}(v)=d_1-1,\deg_{in,T}(w)=d_2,v\ne w]V_T(v,w)\\+[\deg_{out,T}(v)=d_1,\deg_{in,T}(w)=d_2,v\ne w]V_T(v,w).
\end{multline*}
Recall that $v$ and $w$ are selected independently with probabilities $\frac{\deg_{T,out}(v)+\dout}{T+\dout N+A_{out}}$
and $\frac{\deg_{T,in}(w)+\din}{T+\din N+A_{in}}$ correspondingly.
\begin{multline*}
\E(X(T+1+t_0,d_1,d_2|G_T,(\dagger\dagger))=
X(T+t_0,d_1,d_2)\left(1-\frac{d_2+\din}{T+\din N+A_{in}}\right)\left(1-\frac{d_1+\dout}{T+\dout N+A_{out}}\right)\\+
X(T+t_0,d_1,d_2-1)\frac{d_2-1+\din}{T+\din N+A_{in}}\left(1-\frac{d_1+\dout}{T+\dout N+A_{out}}\right)\\+
X(T+t_0,d_1-1,d_2)\frac{d_1-1+\dout}{T+\dout N+A_{out}}\left(1-\frac{d_2+\din}{T+\din N+A_{in}}\right)\\+
\Bigg(X(T+t_0,d_1-1,d_2-1)+n_{out}(T+t_0,d_1-1)n_{in}(T+t_0,d_2-1)\\-
\sum_{v=1}^{N+n_0}[\deg_{T,out}(v)=d_1-1,\deg_{T,in}(v)=d_2-1]\Bigg)\frac{d_1-1+\dout}{T+\dout N+A_{out}}\frac{d_2-1+\din}{T+\din N+A_{in}}
\end{multline*}
\item The case $(\ddagger)$ is symmetrical to $(\dagger)$ with in- and out-degrees exchanged.
\end{itemize}
Finally, noting that $X(T+t_0,d_1-1,d_2-1)=O(T)$, $\sum_{v=1}^{N+n_0}[\deg_{T,out}(v)=d_1-1,\deg_{T,in}(v)=d_2-1]=O(T)$
and ignoring all terms that are $O_{d_1,d_2}(1/T)$, we obtain
\begin{multline}\label{exrecur}
E_X(T+1,N,d_1,d_2)=
[N>0]\frac{\alpha}{\alpha+\gamma}\frac{N}{T+1}\Bigg(
E_X(T,N-1,d_1,d_2)\left(1-\frac{d_2+\din}{T+\din(N-1)+A_{in}}\right)\\+
\frac{d_2-1+\din}{T+\din(N-1)+A_{in}}[d_1=1]E_{in,d_2-1}(T,N-1)+\frac{d_2-1+\din}{T+\din(N-1)+A_{in}}E_X(T,N-1,d_1,d_2-1)\Bigg)\\+
[N<T+1]\frac{T+1-N}{T+1}\Bigg(
E_X(T,N,d_1,d_2)\left(1-\frac{d_2+\din}{T+\din N+A_{in}}-\frac{d_1+\dout}{T+\dout N+A_{out}}\right)\\+
E_X(T,N,d_1,d_2-1)\frac{d_2-1+\din}{T+\din N+A_{in}}+
E_X(T,N,d_1-1,d_2)\frac{d_1-1+\dout}{T+\dout N+A_{out}}\\+
E_{out,d_1-1}(T,N)E_{in,d_2-1}(T,N)\frac{d_1-1+\dout}{T+\dout N+A_{out}}\frac{d_2-1+\din}{T+\din N+A_{in}}\Bigg)\\+
[N>0]\frac{\gamma}{\alpha+\gamma}\frac{N}{T+1}\Bigg(
E_X(T,N-1,d_1,d_2)\left(1-\frac{d_1+\dout}{T+\dout(N-1)+A_{out}}\right)\\+
\frac{d_1-1+\dout}{T+\dout(N-1)+A_{out}}[d_2=1]E_{out,d_1-1}(T,N)+\frac{d_1-1+\dout}{T+\din(N-1)+A_{in}}E_X(T,N-1,d_1-1,d_2)\Bigg)\\
+O_{d_1,d_2}\left(\frac1T\right).
\end{multline}
Here $E_{out,d}(T,N)$ and $E_{in,d}(T,N)$ are defined similarly to the proof of Theorem 1 related to out- and in-degrees
correspondingly (thus, $E_d(T,N)$ from the proof of Theorem 1 is $E_{in,d}(T,N)$; $E_{out,d}$ is the same with exchanged
values $\alpha\leftrightarrow\gamma$, $\din\leftrightarrow\dout$).

According to Lemma \ref{expect1},
$E_{out,d}(T,N)=Tf_{out,d}\left(\frac NT\right)+O(1)$ and $E_{in,d}(T,N)=Tf_{in,d}\left(\frac NT\right)+O(1)$,
where
$$f_{in,d}(x)=x\sum_{i=0}^1 [d\ge i]p_{in,i}\frac{\Gamma\left(i+\din+\frac1\cin\right)}{\cin\Gamma(i+\din)}
\frac{\Gamma(d+\din)}{\Gamma\left(d+1+\din+\frac1\cin\right)},$$
$$f_{out,d}(x)=x\sum_{i=0}^1 [d\ge i]p_{out,i}\frac{\Gamma\left(i+\dout+\frac1\cout\right)}{\cout\Gamma(i+\dout)}
\frac{\Gamma(d+\dout)}{\Gamma\left(d+1+\dout+\frac1\cout\right)},$$
(these representations assume $\din>0$, $\dout>0$, $\cin>0$ and $\cout>0$, but are more convenient to work with
than universal ones from Lemma \ref{expect1}).

\subsection{Some definite integrals}
Define
$$
I_1(c_1,c_2,\xi_1,\xi_2,\xi_3,\xi_4)=\iint_{0\le v^{c_1}\le w^{c_2}\le 1}v^{\xi_1-1}(1-v)^{\xi_2}w^{\xi_3-1}(1-w)^{\xi_4}dvdw
$$
for $c_1>0,c_2>0,\xi_1>0,\xi_2\ge0,\xi_3>0,\xi_4\ge0$.
\begin{lemma}\label{i1recur} The following recurrent equations for $I_1$ hold:
\begin{multline*}
(c_2(\xi_1+\xi_2)+c_1(\xi_3+\xi_4))I_1(c_1,c_2,\xi_1,\xi_2,\xi_3,\xi_4)=c_1\xi_4I_1(c_1,c_2,\xi_1,\xi_2,\xi_3,\xi_4-1)+c_2\xi_2I_1(c_1,c_2,\xi_1,\xi_2-1,\xi_3,\xi_4)
\\\mbox{ for }\xi_2\ge1,\xi_4\ge1;
\end{multline*}
$$
(c_2(\xi_1+\xi_2)+c_1\xi_3)I_1(c_1,c_2,\xi_1,\xi_2,\xi_3,0)=c_1\Beta(\xi_1,\xi_2+1)+c_2\xi_2I_1(\xi_1,\xi_2-1,\xi_3,0)
\mbox{ for }\xi_2\ge1;
$$
$$
(c_2\xi_1+c_1(\xi_3+\xi_4))I_1(c_1,c_2,\xi_1,0,\xi_3,\xi_4)=c_1\xi_4I_1(c_1,c_2,\xi_1,0,\xi_3,\xi_4-1)
\mbox{ for }\xi_4\ge1;
$$
$$
I_1(c_1,c_2,\xi_1,0,\xi_3,0)=\frac{c_1}{\xi_1(c_2\xi_1+c_1\xi_3)}.
$$
\end{lemma}
\begin{proof}
$$
I_1(c_1,c_2,\xi_1,\xi_2,\xi_3,\xi_4)=\int_0^1w^{\xi_3-1}(1-w)^{\xi_4}\left(\int_0^{w^{c_2/c_1}}v^{\xi_1-1}(1-v)^{\xi_2}dv\right)dw.
$$
Let $\xi_2\ge1$. Integrate by parts the inner integral, noting that $v=1-(1-v)$:
\begin{multline*}
\xi_1\int_0^{w^{c_2/c_1}}v^{\xi_1-1}(1-v)^{\xi_2}dv=\left(v^{\xi_1}(1-v)^{\xi_2}\right)\Bigg|_0^{w^{c_2/c_1}}+
\xi_2\int_0^{w^{c_2/c_1}}v^{\xi_1}(1-v)^{\xi_2-1}dv\\=w^{\xi_1c_2/c_1}\left(1-w^{c_2/c_1}\right)^{\xi_2}+
\xi_2\int_0^{w^{c_2/c_1}}v^{\xi_1-1}(1-v)^{\xi_2-1}dv-\xi_2\int_0^{w^{c_2/c_1}}v^{\xi_1-1}(1-v)^{\xi_2}dv;
\end{multline*}
\begin{equation}\label{i1recur.inner}
(\xi_1+\xi_2)I_1(c_1,c_2,\xi_1,\xi_2,\xi_3,\xi_4)=\int_0^1w^{\xi_1c_2/c_1+\xi_3-1}(1-w)^{\xi_4}\left(1-w^{c_2/c_1}\right)^{\xi_2}dw+
\xi_2I_1(c_1,c_2,\xi_1,\xi_2-1,\xi_3,\xi_4).
\end{equation}
Let $\xi_4\ge1$. Integrate by parts the outer integral:
\begin{multline*}\frac{d}{dw}\left((1-w)^{\xi_4}\left(\int_0^{w^{c_2/c_1}}v^{\xi_1-1}(1-v)^{\xi_2}dv\right)\right)=
-\xi_4(1-w)^{\xi_4-1}\left(\int_0^{w^{c_2/c_1}}v^{\xi_1-1}(1-v)^{\xi_2}dv\right)\\+
(1-w)^{\xi_4}\frac{c_2}{c_1}w^{c_2/c_1-1}\left(w^{c_2/c_1}\right)^{\xi_1-1}\left(1-w^{c_2/c_1}\right)^{\xi_2},
\end{multline*}
\begin{multline}\label{i1recur.outer}
\xi_3I_1(c_1,c_2,\xi_1,\xi_2,\xi_3,\xi_4)=\xi_4(I_1(c_1,c_2,\xi_1,\xi_2,\xi_3,\xi_4-1)-I_1(c_1,c_2,\xi_1,\xi_2,\xi_3,\xi_4))\\-
\frac{c_2}{c_1}\int_0^1w^{\xi_3+\xi_1c_2/c_1-1}(1-w)^{\xi_4}\left(1-w^{c_2/c_1}\right)^{\xi_2}dw.
\end{multline}
The first equality of the lemma ($\xi_2\ge1$ and $\xi_4\ge1$) follows from combining \eqref{i1recur.inner}
and \eqref{i1recur.outer}.
For $\xi_2\ge1$ and $\xi_4=0$, \eqref{i1recur.inner} holds, but integration by parts yields an additional term
instead of \eqref{i1recur.outer}:
\begin{multline*}
\xi_3I_1(c_1,c_2,\xi_1,\xi_2,\xi_3,0)=\int_0^1 v^{\xi_1-1}(1-v)^{\xi_2}dv-\frac{c_2}{c_1}\int_0^1 w^{\xi_3+\xi_1c_2/c_1-1}(1-w^{c_2/c_1})^{\xi_2}dw\\
=\Beta(\xi_1,\xi_2+1)-\frac{c_2}{c_1}\left((\xi_1+\xi_2)I_1(c_1,c_2,\xi_1,\xi_2,\xi_3,0)-\xi_2I_1(c_1,c_2,\xi_1,\xi_2-1,\xi_3,0)\right).
\end{multline*}
Finally, the case $\xi_2=0$ is same as $\xi_2\ge1$ without the last term in \eqref{i1recur.inner}.
\end{proof}

Define
$$
I_2(c_1,c_2,\xi_1,\xi_2,\xi_3,\xi_4,\xi_5)=\iint_{0\le v^{c_1}\le w^{c_2}\le1}v^{\xi_1-1}(1-v)^{\xi_2}w^{\xi_3-1}(1-w)^{\xi_4}
\left(\int_{w^{1/c_1}}^1t^{\xi_5-1}dt\right)dvdw
$$
for $c_1>0,c_2>0,\xi_1>0,\xi_2\ge0,\xi_3>0,\xi_4\ge0,\xi_3+\xi_5/\cin>0$. (Note: $\xi_5$ can be zero or negative, in this case
the inner integral is not defined for $w=0$; the conditions ensure that the (improper) outer integral converges).
\begin{lemma}\label{i2recur} The following recurrent equations for $I_2$ hold:
\begin{multline*}
(c_2(\xi_1+\xi_2)+c_1(\xi_3+\xi_4))I_2(c_1,c_2,\xi_1,\xi_2,\xi_3,\xi_4,\xi_5)=
I_1\left(c_1,c_2,\xi_1,\xi_2,\xi_3+\frac{\xi_5}{c_1},\xi_4\right)\\
+c_1\xi_4I_2(c_1,c_2,\xi_1,\xi_2,\xi_3,\xi_4-1,\xi_5)
+c_2\xi_2I_2(c_1,c_2,\xi_1,\xi_2-1,\xi_3,\xi_4,\xi_5)
\mbox{ for }\xi_2\ge1,\xi_4\ge1;
\end{multline*}
\begin{multline*}
(c_2(\xi_1+\xi_2)+c_1\xi_3)I_2(c_1,c_2,\xi_1,\xi_2,\xi_3,0,\xi_5)=
I_1\left(c_1,c_2,\xi_1,\xi_2,\xi_3+\frac{\xi_5}{c_1},0\right)\\
+c_2\xi_2I_2(c_1,c_2,\xi_1,\xi_2-1,\xi_3,0,\xi_5)
\mbox{ for }\xi_2\ge1;
\end{multline*}
\begin{multline*}
(c_2\xi_1+c_1(\xi_3+\xi_4))I_2(c_1,c_2,\xi_1,0,\xi_3,\xi_4,\xi_5)=
I_1\left(c_1,c_2,\xi_1,0,\xi_3+\frac{\xi_5}{c_1},\xi_4\right)\\
+c_1\xi_4I_2(c_1,c_2,\xi_1,0,\xi_3,\xi_4-1,\xi_5)
\mbox{ for }\xi_4\ge1;
\end{multline*}
\begin{equation*}
(c_2\xi_1+c_1\xi_3)I_2(c_1,c_2,\xi_1,0,\xi_3,0,\xi_5)=
I_1\left(c_1,c_2,\xi_1,0,\xi_3+\frac{\xi_5}{c_1},0\right).
\end{equation*}
\end{lemma}
\begin{proof}
$$I_2(c_1,c_2,\xi_1,\xi_2,\xi_3,\xi_4,\xi_5)=\int_0^1 w^{\xi_3-1}(1-w)^{\xi_4}\left(\int_0^{w^{c_2/c_1}}v^{\xi_1-1}(1-v)^{\xi_2}dv\right)
\left(\int_{w^{1/c_1}}^1t^{\xi_5-1}dt\right)dw.$$
Let $\xi_2\ge1$ and $\xi_4\ge1$.
Similar to \eqref{i1recur.inner} from the previous lemma, integration by parts of $\int\dots dv$ gives
\begin{multline}\label{i2recur.inner}
(\xi_1+\xi_2)I_2(c_1,c_2,\xi_1,\xi_2,\xi_3,\xi_4,\xi_5)=\xi_2 I_2(c_1,c_2,\xi_1,\xi_2-1,\xi_3,\xi_4,\xi_5)\\
+\int_0^1 w^{\xi_1c_2/c_1+\xi_3-1}(1-w)^{\xi_4}\left(1-w^{c_2/c_1}\right)^{\xi_2}\int_{w^{1/c_1}}^1t^{\xi_5-1}dtdw.
\end{multline}
Integration by parts of $\int\dots dw$ is similar to \eqref{i2recur.inner} from the previous lemma with
an additional term from $\frac{d}{dw}\int_{w^{1/c_1}}^1t^{\xi_5-1}dt=-\frac1{c_1}w^{\xi_5/c_1-1}$:
\begin{multline}\label{i2recur.outer}
(\xi_3+\xi_4)I_2(c_1,c_2,\xi_1,\xi_2,\xi_3,\xi_4,\xi_5)=\xi_4I_2(c_1,c_2,\xi_1,\xi_2,\xi_3,\xi_4-1,\xi_5)\\
-\frac{c_2}{c_1}\int_0^1 w^{\xi_3+\xi_1c_2/c_1-1}(1-w)^{\xi_4}\left(1-w^{c_2/c_1}\right)^{\xi_2}\int_{w^{1/c_1}}^1t^{\xi_5-1}dtdw\\
+\frac1{c_1}\int_0^1 w^{\xi_3+\xi_5/c_1-1}(1-w)^{\xi_4}\left(\int_0^{w^{c_2/c_1}}v^{\xi_1-1}(1-v)^{\xi_2}dv\right)dw.
\end{multline}
The first equality of the lemma follows from combining \eqref{i2recur.inner} with \eqref{i2recur.outer} and recalling the definition of $I_1$.
The case $\xi_2=0$ is similar to \eqref{i2recur.inner} without the term $I_2(\dots,\xi_2-1,\dots)$; the case $\xi_4=0$ is similar to
\eqref{i2recur.outer} without the term $I_2(\dots,\xi_4-1,\dots)$.
\end{proof}

Define
\begin{multline*}
g_1(d_1,d_2)=\frac\gamma{\alpha+\gamma}\frac{x^2}{1+\dout x}
\sum_{i=0}^1[d_1\ge i+1]\frac{p_{out,i}}{\cin\cout}
\frac{\Gamma(d_1+\dout)}{\Gamma(d_1-i)\Gamma(\dout+i)}\frac{\Gamma(d_2+\din)}{\Gamma(d_2)\Gamma(1+\din)}\\\times
I_1\left(\cin,\cout,\dout+\frac1\cout+i,d_1-i-1,\din+\frac\cout\cin+1,d_2-1\right),
\end{multline*}
\begin{multline*}
g_2(d_1,d_2)=\frac\alpha{\alpha+\gamma}\frac{x^2}{1+\din x}
\sum_{i=0}^1[d_2\ge i+1]\frac{p_{in,i}}{\cin\cout}
\frac{\Gamma(d_1+\dout)}{\Gamma(d_1)\Gamma(1+\dout)}\frac{\Gamma(d_2+\din)}{\Gamma(d_2-i)\Gamma(\din+i)}\\\times
I_1\left(\cout,\cin,\din+\frac1\cin+i,d_2-i-1,\dout+\frac\cin\cout+1,d_1-1\right),
\end{multline*}
\begin{multline*}
g_3(d_1,d_2)={\textstyle\frac{x^2(1-x)}{(1+\din x)(1+\dout x)}}\sum_{i=0}^1\sum_{j=0}^1[d_1\ge i+1,d_2\ge j+1]\frac{p_{out,i}p_{in,j}}{c_{in}c_{out}}
\frac{\Gamma(d_1+\dout)}{\Gamma(d_1-i)\Gamma(\dout+i)}\frac{\Gamma(d_2+\din)}{\Gamma(d_2-j)\Gamma(\din+j)}\\\times
\Bigg(I_2\left(\cin,\cout,\dout+\frac1\cout+i,d_1-i-1,\din+\frac\cout\cin+1+j,d_2-j-1,1-\cin-\cout\right)\\
+I_2\left(\cout,\cin,\din+\frac1\cin+j,d_2-j-1,\dout+\frac\cin\cout+1+i,d_1-i-1,1-\cin-\cout\right)\Bigg)
\end{multline*}
for $d_1\ge1$, $d_2\ge1$. Define $g_1(0,d_2)=g_1(d_1,0)=0$.

Lemma \ref{i1recur} implies the following recurrent equations for $g_1,g_2$ and $d_1\ge1,d_2\ge1$:
\begin{multline}\label{g1recur}
(c_{in}(d_2+\din)+c_{out}(d_1+\dout)+1)g_1(d_1,d_2)=c_{in}(d_2-1+\din)g_1(d_1,d_2-1)+c_{out}(d_1-1+\dout)g_1(d_1-1,d_2)\\+
[d_2=1]\frac\gamma{\alpha+\gamma}x\frac{d_1-1+\dout}{1+\dout x}f_{out,d_1-1}(x);
\end{multline}
\begin{multline}\label{g2recur}
(c_{in}(d_2+\din)+c_{out}(d_1+\dout)+1)g_2(d_1,d_2)=c_{in}(d_2-1+\din)g_2(d_1,d_2-1)+c_{out}(d_1-1+\dout)g_2(d_1-1,d_2)\\+
[d_1=1]\frac\alpha{\alpha+\gamma}x\frac{d_2-1+\din}{1+\din x}f_{in,d_2-1}(x).
\end{multline}
Lemma \ref{i2recur} and the fact that
\begin{multline*}
\Beta(\xi_1,\xi_2+1)\Beta(\xi_3,\xi_4+1)\\=
\iint_{0\le v^{\cin}\le w^{\cout}\le 1}v^{\xi_1-1}(1-v)^{\xi_2}w^{\xi_3-1}(1-w)^{\xi_4}dvdw
+\iint_{0\le w^{\cout}\le v^{\cin}\le 1}v^{\xi_1-1}(1-v)^{\xi_2}w^{\xi_3-1}(1-w)^{\xi_4}dvdw\\=
I_1(\cin,\cout,\xi_1,\xi_2,\xi_3,\xi_4)+I_1(\cout,\cin,\xi_3,\xi_4,\xi_1,\xi_2)
\end{multline*}
imply that for $d_1\ge1,d_2\ge1$
\begin{multline}\label{g3recur}
(c_{in}(d_2+\din)+c_{out}(d_1+\dout)+1)g_3(d_1,d_2)=c_{in}(d_2-1+\din)g_3(d_1,d_2-1)+c_{out}(d_1-1+\dout)g_3(d_1-1,d_2)\\+
(1-x)\frac{d_1-1+\dout}{1+\dout x}\frac{d_2-1+\din}{1+\din x}f_{out,d_1-1}(x)f_{in,d_2-1}(x).
\end{multline}

Finally, define
\begin{equation}\label{g1def}
g(d_1,d_2)=g(x,d_1,d_2)=g_1(d_1,d_2)+g_2(d_1,d_2)+g_3(d_1,d_2).
\end{equation}
Then \eqref{g1recur}, \eqref{g2recur} and \eqref{g3recur} imply the recurrent equation
\begin{multline}\label{grecur}
(c_{in}(d_2+\din)+c_{out}(d_1+\dout)+1)g(d_1,d_2)=c_{in}(d_2-1+\din)g(d_1,d_2-1)+c_{out}(d_1-1+\dout)g(d_1-1,d_2)\\+
[d_2=1]\frac\gamma{\alpha+\gamma}x\frac{d_1-1+\dout}{1+\dout x}f_{out,d_1-1}(x)+
[d_1=1]\frac\alpha{\alpha+\gamma}x\frac{d_2-1+\din}{1+\din x}f_{in,d_2-1}(x)\\+
(1-x)\frac{d_1-1+\dout}{1+\dout x}\frac{d_2-1+\din}{1+\din x}f_{out,d_1-1}(x)f_{in,d_2-1}(x).
\end{multline}

\begin{lemma}\label{exasymp} $E_X(T,N,d_1,d_2)=g\left(\frac{N}{T},d_1,d_2\right)T+O_{d_1,d_2}(1)$.
\end{lemma}
\begin{proof} The proof is essentially same as the proof of Lemma \ref{expect1}.

We use induction by $d_1+d_2$. Both sides are zero when $d_1=0$ or $d_2=0$. For $d_1\ge1$ and $d_2\ge1$
note that $g$ as a function in $x\in[0,1]$ is analytical, so $g(x,d_1,d_2)=O_{d_1,d_2}(1)$,
$g'(x,d_1,d_2)=O_{d_1,d_2}(1)$, $g''(x,d_1,d_2)=O_{d_1,d_2}(1)$ (we use derivatives with respect to $x$).
For fixed $d_1$ and $d_2$ we use induction by $T$ to prove that
$$
\left|E_X(T,N,d_1,d_2)-g\left(\frac NT,d_1,d_2\right)\right|\le C=C(d_1,d_2),
$$
where the constant $C(d_1,d_2)$ will be selected later, assuming that this inequality holds for $d_1-1,d_2$ and for $d_1,d_2-1$ with all $T$.
Induction base $T\le T_0(d_1,d_2)$ (where the value $T_0$ will be selected later) is trivial.
Now assume that the bound for $T$ is proved and consider $T+1$.
Let $0\le N\le T+1$ and $x=\frac{N}{T+1}$. If $N\ne T+1$, the inductive hypothesis and the Taylor formula
imply that
\begin{multline*}
E_X(T,N,d_1,d_2)=Tg\left(x+\frac xT,d_1,d_2\right)+\theta_1=Tg(x,d_1,d_2)+xg'(x,d_1,d_2)+\frac{x^2}{2T}g''(\xi,d_1,d_2)+\theta_1\\=
Tg(x,d_1,d_2)+xg'(x,d_1,d_2)+O_{d_1,d_2}\left(\frac1T\right)+\theta_1,
\end{multline*}
where $|\theta_1|\le C$. Similarly,
$E_X(T,N,d_1-1,d_2)=Tg(x,d_1-1,d_2)+O_{d_1,d_2}(1)$ (even for $d_1=1$) and $E_X(T,N,d_1,d_2-1)=Tg(x,d_1,d_2-1)+O_{d_1,d_2}(1)$
(even for $d_2=1$).
If $N\ne0$, we have for the same reasons
\begin{multline*}
E_X(T,N-1,d_1,d_2)=Tg\left(x-\frac{1-x}T,d_1,d_2\right)+\theta_2\\=Tg(x,d_1,d_2)-(1-x)g'(x,d_1,d_2)+\frac{(1-x)^2}{2T}g''(\xi,d_1,d_2)+\theta_2\\
=Tg(x,d_1,d_2)-(1-x)g'(x,d_1,d_2)+O\left(\frac{1-x}T\right)+\theta_2,
\end{multline*}
where $|\theta_2|\le C$. Similarly,
$E_X(T,N-1,d_1-1,d_2)=Tg(x,d_1-1,d_2)+O_{d_1,d_2}(1)$ and $E_X(T,N-1,d_1,d_2-1)=Tg(x,d_1,d_2-1)+O_{d_1,d_2}(1)$.

Substitute these representations in \eqref{exrecur}. The sum of terms of order $T$ in right-hand side is $Tg(x,d_1,d_2)$,
terms of order $1$ are exactly as in \eqref{grecur} and give $g(x,d_1,d_2)$ in total. Thus,
\begin{multline*}
E_X(T+1,N,d_1,d_2)=(T+1)g\left(\frac N{T+1},d_1,d_2\right)
\\+\theta_1(1-x)\left(1-\frac{d_2+\din}{T(1+\din x)+\din x+A_{in}}-\frac{d_1+\dout}{T(1+\dout x)+\dout x+A_{out}}\right)
\\+\theta_2x\left(1-\frac{\frac\alpha{\alpha+\gamma}(d_2+\din)}{T(1+\din x)+\din(x-1)+A_{in}}-\frac{\frac\gamma{\alpha+\gamma}(d_1+\dout)}{T(1+\dout x)+\dout(x-1)+A_{out}}\right)
+O_{d_1,d_2}\left(\frac1T\right).
\end{multline*}
Select $T_0$ such that coefficients in $\theta_i$ are non-negative for all $T\ge T_0$ and $x\in[0,1]$. Then, for any sufficiently large $C$
we have
\begin{multline*}
\left|E_X(T+1,N,d_1,d_2)-(T+1)g\left(\frac N{T+1},d_1,d_2\right)\right|\le C+O_{d_1,d_2}\left(\frac1T\right)\\
{\textstyle-
C\left(\frac{(1-x)(d_2+\din)}{T(1+\din x)+\din x+A_{in}}+\frac{(1-x)(d_1+\dout)}{T(1+\dout x)+\dout x+A_{out}}+
\frac{\frac\alpha{\alpha+\gamma}x(d_2+\din)}{T(1+\din x)+\din(x-1)+A_{in}}+\frac{\frac\gamma{\alpha+\gamma}x(d_1+\dout)}{T(1+\dout x)+\dout(x-1)+A_{out}}\right)}
\le C.
\end{multline*}
\end{proof}

\subsection{Asymptotic behaviour}
\begin{lemma}\label{kappaasymp}
$$
\kappa(c_1,c_2,r,x)=\Gamma(c_1)\Gamma(c_2)-\kappa\left(c_2,c_1,\frac1r,x^{-1/r}\right).
$$
\end{lemma}
\begin{proof}
$$
\int_0^\infty z^{c_1-1}e^{-z\tau^r}dz=(\tau^r)^{-c_1}\int_0^\infty(z\tau^r)^{c_1-1}e^{-z\tau^r}d(z\tau^r)=\tau^{-c_1r}\Gamma(c_1);
$$
$$
\int_0^\infty\int_0^\infty z^{c_1-1}\tau^{c_1r+c_2-1}e^{-\tau-z\tau^r}dzd\tau=\int_0^\infty(\tau^r)^{-c_1}\Gamma(c_1)\tau^{c_1r+c_2-1}e^{-\tau}d\tau=
\Gamma(c_1)\Gamma(c_2);
$$
$$
\Gamma(c_1)\Gamma(c_2)-\kappa(c_1,c_2,r,x)=\int_x^\infty dz\int_0^\infty z^{c_1-1}\tau^{c_1r+c_2-1}e^{-\tau-z\tau^r}d\tau.
$$
Replace variables $\hat z=1/z^{1/r}, \hat\tau=z\tau^r$: $z=1/\hat z^r$, $\tau=\hat\tau^{1/r}\hat z$,
$dz=-r\hat z^{-r-1}d\hat z$, $d\tau=\frac1r\hat z\hat\tau^{1/r-1}d\hat\tau$,
$$
\Gamma(c_1)\Gamma(c_2)-\kappa(c_1,c_2,r,x)=\int_0^{x^{-1/r}}d\hat z\int_0^\infty \hat z^{c_2-1}\hat\tau^{c_1+c_2/r-1}e^{-\hat\tau^{1/r}\hat z-\hat\tau}d\hat\tau.
$$
\end{proof}

\begin{lemma}\label{i1asymp} Let $r>0,\xi_1>0,\xi_3>0$ be constants depending only on the model parameters. Then
$$I_1(c_1,c_1r,\xi_1,d_1,\xi_3,d_2)=
d_1^{-\xi_1}d_2^{-\xi_3}\kappa\left(\xi_1,\xi_3,r,\frac{d_1}{d_2^r}\right)\left(1+O\left(\frac1{d_2^{\min(r,1)}}\right)\right).$$
\end{lemma}
\begin{proof}
Replace variables: let $v=w^{r}z$.
\begin{multline*}
\iint_{0\le v\le w^{r}\le 1}v^{\xi_1-1}(1-v)^{d_1}w^{\xi_3-1}(1-w)^{d_2}dvdw\\
=\int_0^1\int_0^1w^{\xi_1r-r+\xi_3-1}z^{\xi_1-1}(1-w^{r}z)^{d_1}(1-w)^{d_2}w^rdzdw
=\int_0^1 z^{\xi_1-1}I_zdz,
\end{multline*}
where $I_z=\int_0^1 w^{\xi_1r+\xi_3-1}(1-w)^{d_2}(1-w^{r}z)^{d_1}dw$.

Since $\ln(1+x)\le x$ for any $x>-1$, we have $(1-w)^{d_2}=\exp(d_2\ln(1-w))\le\exp(-d_2w)$ and
$(1-w^{r}z)^{d_1}\le\exp(-d_1w^{r}z)$, so
$$I_z\le\int_0^\infty w^{\xi_1r+\xi_3-1}e^{-d_1w^{r}z-d_2w}dw=d_2^{-\xi_1r-\xi_3}\int_0^\infty w^{\xi_1r+\xi_3-1}e^{-d_1\left(\frac{w}{d_2}\right)^rz-w}dw.$$

If $r\ge1$, let $\tau=\frac w{1-w}$, then $w=\frac{\tau}{1+\tau}$, $dw=\frac{d\tau}{(1+\tau)^2}$,
$$
I_z=\int_0^\infty\frac{\tau^{\xi_1r+\xi_3-1}}{(1+\tau)^{\xi_1r+\xi_3+1+d_2}}\left(1-z\left(\frac{\tau}{1+\tau}\right)^r\right)^{d_1}d\tau
$$
Comparing derivatives, it is easy to see that $(1+\tau)^r\ge 1+\tau^r$ for $r\ge1$. Thus,
$$1-z\left(\frac{\tau}{1+\tau}\right)^r\ge1-\frac{z\tau^r}{1+\tau^r}\ge1-\frac{z\tau^r}{1+z\tau^r}=\frac1{1+z\tau^r}\ge\frac1{\exp(z\tau^r)},$$
\begin{multline*}I_z\ge\int_0^\infty\tau^{\xi_1r+\xi_3-1}e^{-(\xi_1r+\xi_3+1+d_2)\tau-d_1z\tau^r}d\tau
\ge(\xi_1r+\xi_3+1+d_2)^{-\xi_1r-\xi_3}\int_0^\infty\tau^{\xi_1r+\xi_3-1}e^{-\tau-d_1z\left(\frac{\tau}{d_2}\right)^r}d\tau\\=
d_2^{-\xi_1r-\xi_3}\left(1+O\left(\frac1{d_2}\right)\right)\int_0^\infty\tau^{\xi_1r+\xi_3-1}e^{-\tau-\frac{d_1}{d_2^r}z\tau^r}d\tau.\end{multline*}

If $r\le1$, let $\tau=\frac w{(1-w^r)^{1/r}}$, then $w=\frac \tau{(1+\tau^r)^{1/r}}$, $w^{r-1}dw=\tau^{r-1}\frac{d\tau}{(1+\tau^r)^2}$,
$$
I_z=\int_0^\infty\frac{\tau^{\xi_1r+\xi_3-1}}{(1+\tau^r)^{\xi_1+\xi_3/r+1}}\left(1-\frac\tau{(1+\tau^r)^{1/r}}\right)^{d_2}\left(1-\frac{z\tau^r}{1+\tau^r}\right)^{d_1}d\tau.
$$
In this case, $1/r\ge1$, so $(1+\tau^r)^{1/r}\ge1+(\tau^r)^{1/r}=1+\tau$, $1-\frac\tau{(1+\tau^r)^{1/r}}\ge\frac1{1+\tau}\ge\exp(-\tau)$,
$1-\frac{z\tau^r}{1+\tau^r}\ge\frac1{1+z\tau^r}\ge\exp(-z\tau^r)$,
$$
I_z\ge\int_0^\infty\tau^{\xi_1r+\xi_3-1}e^{-d_2\tau-(zd_1+\xi_1+\xi_3/r+1)\tau^r}d\tau
$$
If $zd_1\ge d_2^r$, then $O\left(\frac1{zd_1}\right)=O\left(\frac1{d_2^r}\right)$, so
\begin{multline*}
I_z\ge\left(1+\frac{\xi_1+\xi_3/r+1}{zd_1}\right)^{-\xi_1-\xi_3/r}\int_0^\infty\tau^{\xi_1r+\xi_3-1}e^{-\frac{d_2\tau}{(1+(\xi_1+\xi_3/r+1)/(zd_1))^{1/r}}-zd_1\tau^r}d\tau
\\\ge\left(1+O\left(\frac1{zd_1}\right)\right)\int_0^\infty\tau^{\xi_1r+\xi_3-1}e^{-d_2\tau-zd_1\tau^r}d\tau=
\left(1+O\left(\frac1{d_2^r}\right)\right)d_2^{-\xi_1r-\xi_3}\int_0^\infty\tau^{\xi_1r+\xi_3-1}e^{-\tau-\frac{d_1}{d_2^r}z\tau^r}d\tau.
\end{multline*}
If $zd_1<d_2^r$, then $\exp\left(-(\xi_1+\xi_3/r+1)\frac{\tau^r}{d_2^r}\right)\ge1-(\xi_1+\xi_3/r+1)\frac{\tau^r}{d_2^r}$,
$$
I_z\ge d_2^{-\xi_1r-\xi_3}\int_0^\infty\tau^{\xi_1r+\xi_3-1}e^{-\tau-zd_1\frac{\tau^r}{d_2^r}}\left(1-(\xi_1+\xi_3/r+1)\frac{\tau^r}{d_2^r}\right)d\tau,
$$
$\int_0^\infty\tau^{c}e^{-\tau-\tau^r(zd_1/d_2^r)}d\tau=\Theta(1)$ for $c=\xi_1r+\xi_3-1$ and for $c=\xi_1r+\xi_3+r-1$, so
$$
I_z\ge d_2^{-\xi_1r-\xi_3}\left(1+O\left(\frac1{d_2^r}\right)\right)\int_0^\infty\tau^{\xi_1r+\xi_3-1}e^{-\tau-zd_1\frac{\tau^r}{d_2^r}}d\tau.
$$

Integration by $z$ completes the proof.
\end{proof}
\begin{lemma}\label{kappaasymp2} Let $r>0,\xi_1>0,\xi_3>0$ be constants depending only on the model parameters. Then
$$I_1(c_1,c_1r,\xi_1,d_1,\xi_3,d_2)=
d_1^{-\xi_1}d_2^{-\xi_3}\kappa\left(\xi_1,\xi_3,r,\frac{d_1}{d_2^r}\right)\left(1+O\left(\frac1{d_1}\right)+O\left(\frac1{d_2}\right)\right).$$
\end{lemma}
\begin{proof} If $r\ge1$ or $d_1\le Cd_2^r$, the statement follows immediately from the previous lemma. Assume $r<1$ and $d_1>Cd_2^r$,
where the constant $C$ will be selected later.
\begin{multline*}
\iint_{0\le v\le w^{r}\le 1}v^{\xi_1-1}(1-v)^{d_1}w^{\xi_3-1}(1-w)^{d_2}dvdw\\=
\Beta(\xi_1,d_1+1)\Beta(\xi_3,d_2+1)-\iint_{0\le w^{r}\le v\le 1}v^{\xi_1-1}(1-v)^{d_1}w^{\xi_3-1}(1-w)^{d_2}dvdw\\=
\Gamma(\xi_1)\Gamma(\xi_3)d_1^{-\xi_1}d_2^{-\xi_3}\left(1+O\left(\frac1{d_1}\right)+O\left(\frac1{d_2}\right)\right)-
\iint_{0\le w\le v^{1/r}\le1}w^{\xi_3-1}(1-w)^{d_2}v^{\xi_1-1}(1-v)^{d_1}dvdw\\=
\Gamma(\xi_1)\Gamma(\xi_3)d_1^{-\xi_1}d_2^{-\xi_3}\left(1+O\left(\frac1{d_1}\right)+O\left(\frac1{d_2}\right)\right)-
d_1^{-\xi_1}d_2^{-\xi_3}\kappa\left(\xi_3,\xi_1,\frac1r,\frac{d_2}{d_1^{1/r}}\right)\left(1+O\left(\frac1{d_1}\right)\right)
\end{multline*}
Lemma \ref{kappaasymp} implies that
$$\Gamma(\xi_1)\Gamma(\xi_3)-\kappa\left(\xi_3,\xi_1,\frac1r,\frac{d_2}{d_1^{1/r}}\right)=
\kappa\left(\xi_1,\xi_3,r,\frac{d_1}{d_2^r}\right).$$
On the other hand,
$\kappa(c_1,c_2,r,x)=x^{c_1}\int_0^1 z^{c_1-1}dz\int_0^\infty\tau^{c_1r+c_2-1}e^{-\tau-xz\tau^r}d\tau\le x^{c_1}\frac{\Gamma(c_1r+c_2)}{\xi_3}$, so
$$\kappa\left(\xi_3,\xi_1,\frac1r,\frac{d_2}{d_1^{1/r}}\right)\le C^{-\xi_3/r}\frac{\Gamma(\xi_1+\xi_3/r)}{\xi_3}
\le\frac12\Gamma(\xi_1)\Gamma(\xi_3)$$
if $C$ is sufficiently large, so
$O(1)=O\left(\Gamma(\xi_1)\Gamma(\xi_3)-\kappa\left(\xi_3,\xi_1,\frac1r,\frac{d_2}{d_1^{1/r}}\right)\right)$.
\end{proof}

\begin{lemma}\label{ishift} Let $c_1>0,c_2>0,\xi_1>0,\xi_3>0,\xi_5>-\xi_3\cin$ be constants depending only on the model parameters. Let $d_1\ge1$, $d_2\ge1$. Then
$$
I_1(c_1,c_2,\xi_1,d_1\pm1,\xi_3,d_2)=I_1(c_1,c_2,\xi_1,d_1,\xi_3,d_2)\left(1+O\left(\frac1{d_1}\right)\right),
$$
$$
I_1(c_1,c_2,\xi_1,d_1,\xi_3,d_2\pm1)=I_1(c_1,c_2,\xi_1,d_1,\xi_3,d_2)\left(1+O\left(\frac1{d_2}\right)\right).
$$
$$
I_2(c_1,c_2,\xi_1,d_1\pm1,\xi_3,d_2,\xi_5)=I_2(c_1,c_2,\xi_1,d_1,\xi_3,d_2,\xi_5)\left(1+O\left(\frac1{d_1}\right)\right),
$$
$$
I_2(c_1,c_2,\xi_1,d_1,\xi_3,d_2\pm1,\xi_5)=I_2(c_1,c_2,\xi_1,d_1,\xi_3,d_2,\xi_5)\left(1+O\left(\frac1{d_2}\right)\right).
$$
\end{lemma}
\begin{proof}
It is easy to see from the definition of $I_1$ that $I_1(c_1,c_1r,\xi_1,d_1,\xi_3,d_2)\le I_1(c_1,c_1r,\xi_1,d_1-1,\xi_3,d_2)$
and $I_1(c_1,c_1r,\xi_1,d_1,\xi_3,d_2)\le I_1(c_1,c_1r,\xi_1,d_1,\xi_3,d_2-1)$. On the other hand,
Lemma \ref{i1recur} implies
\begin{multline*}
(c_2(\xi_1+d_1)+c_1(\xi_3+d_2))I_1(c_1,c_2,\xi_1,d_1,\xi_3,d_2)=c_1d_2I_1(c_1,c_2,\xi_1,d_1,\xi_3,d_2-1)+c_2d_1I_1(c_1,c_2,\xi_1,d_1-1,\xi_3,d_2)\\
\ge c_1d_2I_1(c_1,c_2,\xi_1,d_1,\xi_3,d_2)+c_2d_1I_1(c_1,c_2,\xi_1,d_1-1,\xi_3,d_2),
\end{multline*}
$$
(c_2\xi_1+c_1\xi_3+c_2d_1)I_1(c_1,c_2,\xi_1,d_1,\xi_3,d_2)\ge c_2d_1I_1(c_1,c_2,\xi_1,d_1-1,\xi_3,d_2),
$$
$$
I_1(c_1,c_2,\xi_1,d_1-1,\xi_3,d_2)\le\left(1+\frac{c_2\xi_1+c_1\xi_3}{c_2d_1}\right)I_1(c_1,c_2,\xi_1,d_1,\xi_3,d_2),
$$
$$
I_1(c_1,c_2,\xi_1,d_1,\xi_3,d_2)\ge\left(1-\frac{c_2\xi_1+c_1\xi_3}{c_2d_1}\right)I_1(c_1,c_2,\xi_1,d_1-1,\xi_3,d_2).
$$
This proves the first statement of the lemma, other are analogous (for $I_2$, use Lemma \ref{i2recur}
instead of Lemma \ref{i1recur} and note that $I_1\ge0$).
\end{proof}

\begin{lemma}\label{i2asymp} Let $c_1>0,r>0,\xi_1>0,\xi_3>0$ be constants depending only on the model parameters.
If $\xi_5\ne0$, then
\begin{multline*}
I_2(c_1,c_1r,\xi_1,d_1,\xi_3,d_2,\xi_5)=\frac1{\xi_5}d_1^{-\xi_1}\Bigg(
d_2^{-\xi_3}\kappa\left(\xi_1,\xi_3,r,\frac{d_1}{d_2^r}\right)\left(1+O\left(\frac1{d_1}+\frac1{d_2}\right)\right)\\
-d_2^{-\xi_3-\frac{\xi_5}{c_1}}\kappa\left(\xi_1,\xi_3+\frac{\xi_5}{c_1},r,\frac{d_1}{d_2^r}\right)\left(1+O\left(\frac1{d_1}+\frac1{d_2}\right)\right)
\Bigg).
\end{multline*}
If $\xi_5=0$, then
\begin{multline*}
I_2(c_1,c_1r,\xi_1,d_1,\xi_3,d_2,0)=\frac1{c_1}d_1^{-\xi_1}d_2^{-\xi_3}
\left(\kappa\left(\xi_1,\xi_3,r,\frac{d_1}{d_2^r}\right)\ln d_2-\frac{\partial}{\partial\xi_3}\kappa\left(\xi_1,\xi_3,r,\frac{d_1}{d_2^r}\right)\right)
\\+O\left(d_1^{-\xi_1}d_2^{-\xi_3}\left(\frac{\ln d_2}{d_1}+\frac1{d_2}\right)\right).
\end{multline*}
\end{lemma}
\begin{proof} The first statement follows from the equality
$$
I_2(c_1,c_1r,\xi_1,d_1,\xi_3,d_2,\xi_5)=\frac{I_1(c_1,c_1r,\xi_1,d_1,\xi_3,d_2)-I_1(c_1,c_1r,\xi_1,d_1,\xi_3+\xi_5/c_1,d_2)}{\xi_5},
$$
and Lemma \ref{kappaasymp2}.

The case $\xi_5=0$ is special.
$$
I_2(c_1,c_1r,\xi_1,d_1,\xi_3,d_2,0)=\frac1{c_1}\iint_{0\le v^{c_1}\le w^{c_1r}\le 1}v^{\xi_1-1}(1-v)^{d_1}w^{\xi_3-1}(1-w)^{d_2}(-\ln w)dvdw.
$$

For the upper bound, note that $1-v\le\exp(-v)$, $1-w\le\exp(-w)$:
\begin{multline*}
c_1I_2(c_1,c_1r,\xi_1,d_1,\xi_3,d_2,0)\le\iint_{0\le v^{c_1}\le w^{c_1r}\le1}v^{\xi_1-1}w^{\xi_3-1}e^{-d_1v-d_2w}(-\ln w)dvdw
\\\le\iint_{0\le v\le w^r<\infty}v^{\xi_1-1}w^{\xi_3-1}e^{-d_1v-d_2w}(-\ln w)dvdw
+\iint_{v\ge0,w\ge1}v^{\xi_1-1}w^{\xi_3-1}e^{-d_1v-d_2w}\ln wdvdw;
\end{multline*}
$\int_0^\infty v^{\xi_1-1}e^{-d_1v}dv=d_1^{-\xi_1}\Gamma(\xi_1)$,
$\int_1^\infty w^{\xi_3-1}e^{-d_2w}\ln wdw\le e^{-d_2}\int_1^\infty w^{\xi_3-1}e^{-(w-1)}\ln wdw=O\left(e^{-d_2}\right)$,
so
$$
c_1I_2(c_1,c_1r,\xi_1,d_1,\xi_3,d_2,0)\le\iint_{0\le v\le w^r<\infty}v^{\xi_1-1}w^{\xi_3-1}e^{-d_1v-d_2w}(-\ln w)dvdw
+O\left(d_1^{-\xi_1}e^{-d_2}\right).
$$

For the lower bound, note that
\begin{multline}\label{i2temp}
c_1I_2(c_1,c_1r,\xi_1,d_1,\xi_3,d_2,0)=\left(\int_0^1 v^{\xi_1-1}(1-v)^{d_1}dv\right)\left(\int_0^1 w^{\xi_3-1}(1-w)^{d_2}(-\ln w)dw\right)
\\-\iint_{0\le w\le v^{1/r}\le 1}v^{\xi_1-1}(1-v)^{d_1}w^{\xi_3-1}(1-w)^{d_2}(-\ln w)dvdw;
\end{multline}
$\int_0^1 v^{\xi_1-1}(1-v)^{d_1}dv=\Beta(\xi_1,d_1+1)=\Gamma(\xi_1)d_1^{-\xi_1}\left(1+O\left(\frac1{d_1}\right)\right);$
according to \cite[6.3.18]{specfunc}, we have $\frac{\Gamma'(x)}{\Gamma(x)}=\ln x+O\left(\frac1x\right)$,
so
\begin{multline*}
\int_0^1 w^{\xi_3-1}(1-w)^{d_2}(-\ln w)dw=-\frac\partial{\partial\xi_3}\Beta(\xi_3,d_2+1)=
-\frac{\Gamma'(\xi_3)\Gamma(d_2+1)}{\Gamma(\xi_3+d_2+1)}+\frac{\Gamma(\xi_3)\Gamma(d_2+1)\Gamma'(\xi_3+d_2+1)}{\Gamma(\xi_3+d_2+1)^2}
\\=d_2^{-\xi_3}\left(-\Gamma'(\xi_3)+\Gamma(\xi_3)\ln d_2+O\left(\frac1{d_2}\right)\right);
\end{multline*}
\begin{multline*}
\iint_{0\le w\le v^{1/r}\le 1}v^{\xi_1-1}(1-v)^{d_1}w^{\xi_3-1}(1-w)^{d_2}(-\ln w)dvdw
\\\le\iint_{0\le w\le v^{1/r}\le 1}v^{\xi_1-1}w^{\xi_3-1}e^{-d_1v-d_2w}(-\ln w)dvdw
\\\le\iint_{0\le w\le v^{1/r}<\infty}v^{\xi_1-1}w^{\xi_3-1}e^{-d_1v-d_2w}(-\ln w)dvdw
+\iint_{1\le w\le v^{1/r}<\infty}v^{\xi_1-1}w^{\xi_3-1}e^{-d_1v-d_2w}\ln wdvdw
\\=\iint_{0\le w\le v^{1/r}<\infty}v^{\xi_1-1}w^{\xi_3-1}e^{-d_1v-d_2w}(-\ln w)dvdw+O\left(e^{-d_1-d_2}\right).
\end{multline*}
Since $\int_0^\infty v^{\xi_1-1}e^{-d_1v}dv=\Gamma(\xi_1)d_1^{-\xi_1}$ and $$\int_0^\infty w^{\xi_3-1}e^{-d_2w}(-\ln w)dw
=-\frac\partial{\partial\xi_3}\left(\Gamma(\xi_3)d_2^{-\xi_3}\right)=d_2^{-\xi_3}\left(-\Gamma'(\xi_3)+\Gamma(\xi_3)\ln d_2\right),$$
we have
\begin{multline*}
c_1I_2(c_1,c_1r,\xi_1,d_1,\xi_3,d_2,0)\\\ge d_1^{-\xi_1}d_2^{-\xi_3}O\left(\frac{\ln d_2}{d_1}+\frac1{d_2}\right)
+\iint_{0\le v\le w^r<\infty}v^{\xi_1-1}w^{\xi_3-1}e^{-d_1v-d_2w}(-\ln w)dvdw+O\left(e^{-d_1-d_2}\right).
\end{multline*}
Combine lower and upper bound:
$$
c_1I_2(c_1,c_1r,\xi_1,d_1,\xi_3,d_2,0)=
\iint_{0\le v\le w^r<\infty}v^{\xi_1-1}w^{\xi_3-1}e^{-d_1v-d_2w}(-\ln w)dvdw+
d_1^{-\xi_1}d_2^{-\xi_3}O\left(\frac{\ln d_2}{d_1}+\frac1{d_2}\right).
$$

Replace variables: let $v=w^rz$.
\begin{multline*}
\iint_{0\le v\le w^r<\infty}v^{\xi_1-1}w^{\xi_3-1}e^{-d_1v-d_2w}(-\ln w)dvdw
=-\frac\partial{\partial\xi_3}\iint_{0\le v\le w^r<\infty}v^{\xi_1-1}w^{\xi_3-1}e^{-d_1v-d_2w}dvdw\\=
-\frac\partial{\partial\xi_3}\int_0^1 z^{\xi_1-1}dz\int_0^\infty w^{\xi_1r+\xi_3-1}e^{-d_1w^rz-d_2w}dw=
-\frac\partial{\partial\xi_3}\left(d_1^{-\xi_1}d_2^{-\xi_3}\kappa\left(\xi_1,\xi_3,r,\frac{d_1}{d_2^r}\right)\right)
\\=d_1^{-\xi_1}d_3^{-\xi_3}\left(\kappa\left(\xi_1,\xi_3,r,\frac{d_1}{d_2^r}\right)\ln d_2-\frac\partial{\partial\xi_3}\kappa\left(\xi_1,\xi_3,r,\frac{d_1}{d_2^r}\right)\right).
\end{multline*}
\end{proof}

\begin{proof}[Proof of Theorem 3]
Lemma \ref{mainbinom} and Lemma \ref{exasymp}
imply $\E X(d_1,d_2)=g(\alpha+\gamma,d_1,d_2)+O_{d_1,d_2}(1)$. It remains to apply
Lemmas \ref{kappaasymp2}, \ref{ishift} and \ref{i2asymp} to the definition of $g$.
\end{proof}

\section{Concentration for number of edges}
\begin{proof}[Proof of Theorem 4]
We use the Azuma--Hoeffding inequality.
\begin{theorem} \cite{azuma}, \cite{hoeffding} Let $(X_s)_{s=0}^n$ be a martingale with
$|X_{s+1}-X_s|\le\delta$ for $s=0,\dots,n-1$, and $x>0$. Then
$$
P\left(|X_n-X_0|\ge x\right)\le2\exp\left(-\frac{x^2}{2\delta^2n}\right).
$$
\end{theorem}

We fix $d_1,d_2,t$ and denote $X=X(t,d_1,d_2)$. Let $G$ be a random graph in $\mathcal G(t)$; it has $t$ edges,
sorted by the creation time. Let $G^{(s)}$ be a graph with $s$ first edges.
Let $X_s=\E(X|G^{(s)}), s=t_0,\dots,t$. In this sequence $X_{t_0}=\E X$, $X_{t}=X$.
By definition of the probabilistic space, the sequence $X_s$ is a martingale.
We will estimate possible differences between adjacent elements of the sequence.

We fix any $s$ from $0$ to $t-1$. Let $v$ and $w$ be the source and the target of the last edge in $G^{(s+1)}$,
so $v$ and $w$ are random quantities depending on $G$. Note that the pair $(v,w)$ also determines
which one of processes $(\dagger)$, $(\ddagger)$, $(\dagger\dagger)$ was used.
By definition
$$\begin{array}{rcrl}
X_s &=& \sum_{i,j} \Pr(v=i,w=j) & \E(X|G^{(s)},v=i,w=j),\\
X_{s+1} &=& & \E(X|G^{(s)},v=v(G^{(s+1)}),w=w(G^{(s+1)})),
\end{array}$$
where the sum is over all pairs of nodes of $G$. Hence it is clear that
$$
\min_{i,j} \E(X|G^{(s)},v=i,w=j)\le X_s,X_{s+1}\le\max_{i,j}\E(X|G^{(s)},v=i,w=j),
$$
$$
|X_s-X_{s+1}|\le\max_{i,j} \E(X|G^{(s)},v=i,w=j)-\min_\gamma \E(X|G^{(s)},v=i,w=j).
$$
Let $\gamma_1\in\arg\min \E(X|G^{(s)},v=i,w=j)$ and $\gamma_2\in\arg\max \E(X|G^{(s)},v=i,w=j)$.
It is sufficient to prove an upper bound for
$$
\E(X|G^{(s)},v=i_2,w=j_2)-\E(X|G^{(s)},v=i_1,w=j_1).
$$

Replace the initial graph $G_0$ by $G^{(s)}\cup(i_1,j_1)$ and consider
the quantity $E_X$ in that model; denote it as $E_1(T,N,d_1,d_2)$.
Similarly, let $E_2(T,N,d_1,d_2)$ denote the analogue of $E_X$ in the model
with the initial graph $G^{(s)}\cup(i_2,j_2)$. Then
$\E(X|G^{(s)},v=i_1,w=j_1)=\E(E_1(t-s-1,N,d_1,d_2))$ and $\E(X|G^{(s)},v=i_2,w=j_2)=\E(E_2(t-s-1,N,d_1,d_2))$,
where $N$ has binomial distribution with parameters $t-s-1$ and $\alpha+\gamma$.
We cannot use Theorem 3 for $E_1$ and $E_2$ directly because the remainder term of Theorem 3 depends
on the initial graph, but \eqref{exrecur} and \eqref{recur1} are still valid.
Here $A_{in}$ and $A_{out}$ can differ in models
with initial graphs $G^{(s)}\cup(i_1,j_1)$ and $G^{(s)}\cup(i_2,j_2)$ due to
different number of vertices.
However, apriori bounds $0\le n_{in}(t,d)\le t$ and $0\le X(t,d_1,d_2)\le t$
imply $E_i(T,N,d_1,d_2)\frac{d_1+\dout}{T+\dout N+A'_{out}}=
E_i(T,N,d_1,d_2)\frac{d_1+\dout}{T+\dout N+A_{out}}+O_{d_1,d_2}\left(\frac1T\right)$
(and similarly for other terms of \eqref{exrecur} and \eqref{recur1}), so
the difference $E_1(T,N,d_1,d_2)-E_2(T,N,d_1,d_2)$ satisfies to
\eqref{exrecur} with $E_{in,d_2-1}(T,N-1)$ and $E_{out,d_1-1}(T,N-1)$
replaced by the corresponding differences, which in turn satisfy to \eqref{recur1} without
terms $\frac{\alpha N}{(\alpha+\gamma)(T+1)}$ and $\frac{\gamma N}{(\alpha+\gamma)(T+1)}$.

It remains to note that adding an edge $(i_1,j_1)$ changes $X(t,d_1,d_2)$
by at most $d_1+d_2$, so the initial value $E_1(0,0,d_1,d_2)-E_2(0,0,d_1,d_2)$ is at most $2(d_1+d_2)$.
Consequently, $E_1(T,N,d_1,d_2)-E_2(T,N,d_1,d_2)=O_{d_1,d_2}(1)$
similar to Lemma \ref{exasymp}.

Therefore, the sequence $(X_s)$ satisfies the condition of Theorem 5 with $n=t-t_0$ and $\delta=O_{d_1,d_2}(1)$.
Substituting $x=\sqrt{t}\ln t$ in Theorem 5, we obtain Theorem 4.
\end{proof}

\end{document}